\documentclass[11pt]{amsart}
\usepackage{mathrsfs,amssymb}
\usepackage{verbatim}

\usepackage{graphicx}

\newtheorem{theorem}{Theorem}[section]
\newtheorem{proposition}[theorem]{Proposition}
\newtheorem{lemma}[theorem]{Lemma}
\newtheorem{corollary}[theorem]{Corollary}
\newtheorem{conjecture}[theorem]{Conjecture}

\theoremstyle{remark}
\newtheorem{remark}[theorem]{Remark}

 \numberwithin{equation}{section}

\renewcommand{\Re}{\operatorname{Re}}
\renewcommand{\Im}{\operatorname{Im}}

\def \R {{\mathbb R}}

\def \C {{\mathbb C}}
\def \Z {{\mathbb Z}}
\def \Q {{\mathbb Q}}

\newcommand{\fq}{\mathbb{F}_q}

\newcommand{\tr}{\operatorname{tr}}

\newcommand{\area}{\operatorname{area}}

\newcommand{\Norm}{\operatorname{Norm}}

\newcommand{\length}{\operatorname{length}}

\newcommand{\EN}{\mathcal N}

\newcommand{\E}{\mathbb E}

\newcommand{\ave}[1]{\left\langle#1\right\rangle} 

\newcommand{\Var}{\operatorname{Var}}
\newcommand{\Prob}{\operatorname{Prob}}

\newcommand{\var}{\operatorname{Var}}

\renewcommand{\^}{\widehat}

\newcommand{\USp}{\operatorname{USp}}

\newcommand{\Ball}{\operatorname{Sect}} 

\newcommand{\kk}{k}
\newcommand{\Sone}{\mathbb S^1} 
\newcommand{\Hev}{H_{\kk}} 

\newcommand{\ord}{\operatorname{ord}}

\begin{document}

\title{Angles of Gaussian primes}
\author{Ze\'ev Rudnick and Ezra Waxman}
\date{\today}
\address{Raymond and Beverly Sackler School of Mathematical Sciences,
Tel Aviv University, Tel Aviv 69978, Israel}
\email{rudnick@post.tau.ac.il, ezrawaxman@gmail.com}

\begin{abstract}
Fermat showed that every prime $p = 1 \bmod 4$ is a sum of
two squares: $p = a^2 + b^2$. To any of the $8$ possible representations $(a,b)$ we associate an angle
whose tangent is the ratio $b/a$. In 1919 Hecke showed that these angles
are uniformly distributed as $p$ varies, and in the 1950's Kubilius proved uniform distribution in somewhat short
arcs.  We study fine scale statistics of these angles, in particular the variance of the number
of such angles in a short arc. We present a conjecture for this variance,
motivated both by a random matrix model, and by a function field
analogue of this problem, for which we prove an asymptotic form for
the corresponding variance.
\end{abstract}
\maketitle

\tableofcontents

\section{Introduction}
\subsection{Angles of Gaussian primes} 
An odd prime $p$ is a sum of two squares if and only if $p=1\bmod 4$, and in that case there are exactly $8$ representations. 
  Each representation  corresponds to a Gaussian integer $a+ib=\sqrt{p}e^{i\theta_{a,b}}$. We wish to understand the statistics of the resulting angles. 
  
It is useful to formulate the results in terms of prime ideals of the ring of Gaussian integers $\Z[i]$, which is the ring of integers of the   imaginary quadratic field $\Q(i)$.    
The basic infra-structure that we need is complex conjugation $z\mapsto \bar z$,   the norm map $\Norm:\Q(i)^\times\to \Q^\times$, $ \Norm(z) = z\bar z$, and the 
norm one elements 
$$S^1_\Q = \{z\in \Q(i):\Norm(z)=1\} = \Q(i)\cap S^1\;.
$$ 
For a Gaussian number  $\alpha  \in \Q(i)^\times$, we have a direction vector  given by 
$$
u(\alpha):=\left(\frac{\alpha}{\bar\alpha}\right)^2 \in \Sone_\Q  
$$
so that $u(\alpha)=e^{ 4i\theta} $, $\theta=\arg\alpha$. 

Let $\mathfrak p$ be a prime ideal in $\Z[i]$. If $\mathfrak p = \langle \alpha \rangle$ is generated by the Gaussian integer $\alpha$, we associate  a direction vector  $u(\mathfrak p):=u(\alpha)\in \Sone_\Q$. Since all generators of the ideal differ by multiplication by a unit $\Z[i]^\times=\{\pm 1,\pm i\}$, the direction vector $u(\mathfrak p) = e^{ i4\theta_{\mathfrak p}}$ is well-defined on ideals, while the angle   
$\theta_{\mathfrak p} $ is only defined modulo $\pi/2$.  
We can choose $\theta_{\mathfrak p}$ to lie say in $[0,\pi/2)$, corresponding to taking $\alpha=a+ib$, with $a>0$, $b\geq 0$. 

  Hecke \cite{Hecke}   showed that as $\mathfrak p$ varies over prime ideals of $\Z[i]$,   the angles $\theta_{\mathfrak p}$ become uniformly distributed in
$[0,\frac\pi 2)$: For a fixed sector, defined by an interval $I\subseteq [0,\frac \pi 2)$,
\begin{equation}\label{eq:Hecke}
 \frac  {\#\{\Norm\mathfrak p\leq x:  \theta_{\mathfrak p}\in
I\}}{\#\{\Norm \mathfrak p\leq x\}} \sim \frac{|I|}{\pi/2},\quad x\to \infty
\end{equation}
where $|I|$ is the length of the interval $I$. 

The validity of \eqref{eq:Hecke} for shrinking sectors  was studied by Kubilius and his school \cite{Kubilius 1950, Kubilius 1955, Kovalcik, Maknis1975, Maknis1976, Maknis1977},  obtaining  that \eqref{eq:Hecke} holds for any sector as long as $|I|>x^{-\delta}$ for some $1/4<\delta<1/2$. See also \cite{Harman Lewis} for existence of prime angles in somewhat smaller sectors without the full force of  \eqref{eq:Hecke}. 
 Assuming the Generalized Riemann Hypothesis (GRH), we know that \eqref{eq:Hecke} holds for intervals  with $\length(I)\gg x^{-1/2+o(1)}$.   
This regime is the limit of what can be expected to hold for individual sectors, because it is easy to see that  there are no Gaussian integers (let alone primes)  in the sector $\{a,b>0: a^2+b^2\leq x, 0<\arctan\frac ba<x^{-1/2}\} $. Hence for smaller sectors we can only hope for a statistical theory, rather than individual results.

To formulate the theory, we introduce some notation: Given $x\gg 1$, let $N$  
be the number of prime ideals $\mathfrak p\subset \Z[i]$ of norm at most $x$:  
$$N:=\#\{\mathfrak p\;{\rm prime}: \Norm \mathfrak p\leq x\}\sim \frac x{\log x}, 
$$
where the asymptotic holds by the Prime Ideal Theorem for $\Q(i)$. 
Given an interval $I_K(\theta)=[\theta-\frac {\pi}{4K},\theta+\frac {\pi}{4K}]$ of length
$\pi/(2K)$ centered at $\theta$,   define a sector 
$$
 \Ball(\theta,x) =\{z \in \C: \Norm(z)=z\bar z \leq x, \arg(z)\in I_K(\theta)\}
$$
of radius $\sqrt{x}$ and opening angle defined by $I_K(\theta)$. 

Given $K\gg 1$, we divide the interval $[0, \pi/2)$ into $K$ disjoint arcs $I_K(\theta_1)$, $\dots$, $I_K(\theta_K)$ of equal length, which in turn define $K$ disjoint sectors $\Ball(\theta_j,x)$,  and study the number of prime angles falling into each such sector.  
  If the sectors are too small, in the sense that the number $K$ of sectors is larger than the number $N$ of angles involved, then the typical such sector will not contain any Gaussian prime. 
   We want to show that in the range   $   K\ll  N^{1-\epsilon}$, 
almost all  sectors with opening angles of size $\approx 1/K$ contain at least one angle $\theta_{\mathfrak p}$, $\Norm(\mathfrak p)\leq x$. We can do so
assuming GRH (for the family of Hecke L-functions):
\begin{theorem}\label{thm aa sectors} 
Assume GRH. Then almost all arcs of length $1/K$ contain at least one  angle $\theta_{\mathfrak p}$ for a prime ideal with $\Norm(\mathfrak p)\leq K(\log K)^{2+o(1)}$. 
\end{theorem}
  Unconditionally, one may use zero-density theorems as in \cite{Maknis1977} to obtain a result with $\Norm(\mathfrak p)<K^{2-\delta}$ for some small $\delta>0$. 

It is surprising that something like Theorem~\ref{thm aa sectors} does not seem to have been considered long ago. It has come up independently in the recent work of Ori Parzanchevski and Peter Sarnak \cite{PR}.

\subsection{The number variance} 
One way to obtain such an ``almost-everywhere'' result is by computing the variance of a suitable counting function. The study of the structure of the variance is the main point of this paper.

Let
\begin{equation}\label{def NKX}
\EN_{K,x}(\theta) = \#\{\mathfrak p \;{\rm prime},\; \Norm\mathfrak p \leq x,\;\theta_{\mathfrak p}\in I_K(\theta)\} 
\end{equation}
be the number of angles $\theta_{\mathfrak p}$ in $I_K(\theta)$.
   
 The expected number is
$$\ave{\EN_{K,x}} : =\int_0^{ \pi/2}\EN_{K,x}(\theta)\frac{d\theta}{\pi/2}
= \frac NK \; .
$$
We wish to study the number variance 
$$
\Var(\EN_{K,x}) =  \int_0^{ \pi/2}\Big|\EN_{K,x}-\ave{\EN_{K,x}} \Big|^2 \frac{d\theta}{\pi/2} \;.
$$

If $N=o(K)$, then for almost all
intervals, we do not have any angles $\theta_{\mathfrak p}$ in the interval
$I_K(\theta)$. We can easily compute the variance in this
``trivial" regime:  
\begin{equation*}
 \var(\EN_{K,x})\sim  \frac NK,\quad N=o(K)    .
 \end{equation*}
 For the interesting range,  
 when $K\ll N^{1-\epsilon}$, we expect:   
\begin{conjecture}\label{sharp conj full regime}
For $1\ll K\ll N^{1-o(1)}$
$$
\var(\EN_{K,x})\sim  \frac{N}{K  }    \min\left (1, 2\; \frac{\log K}{\log N} \right)  .
$$
\end{conjecture}

    For {\em random} angles ($N$ uniform independent points in $[0, \pi/2)$), the variance would be 
    $\sim N /K$. Thus we expect the Gaussian angles to display a marked deviation from randomness, in that there is a crossover  from purely random behaviour   for very short intervals ($K\gg N^{1/2}$), to a saturation for moderately short intervals ($1\ll K\ll N^{1/2}$), where the variance is smaller than that of random angles, so one can say that they display some measure of rigidity.   
See Figure~\ref{fig plot var} for numerical evidence. 
   For an explanation of the underlying rigidity present here and for other deviations from randomness, see \S\ref{sec:repulsion}.
    
    A related saturation effect was previously observed by Bui, Keating and Smith \cite{BKS}, in the context of computing the variance of sums in short intervals  of coefficients of a fixed L-function of higher degree.     
   
     \begin{figure}[h]
\begin{center}
  \includegraphics[width=80mm]
{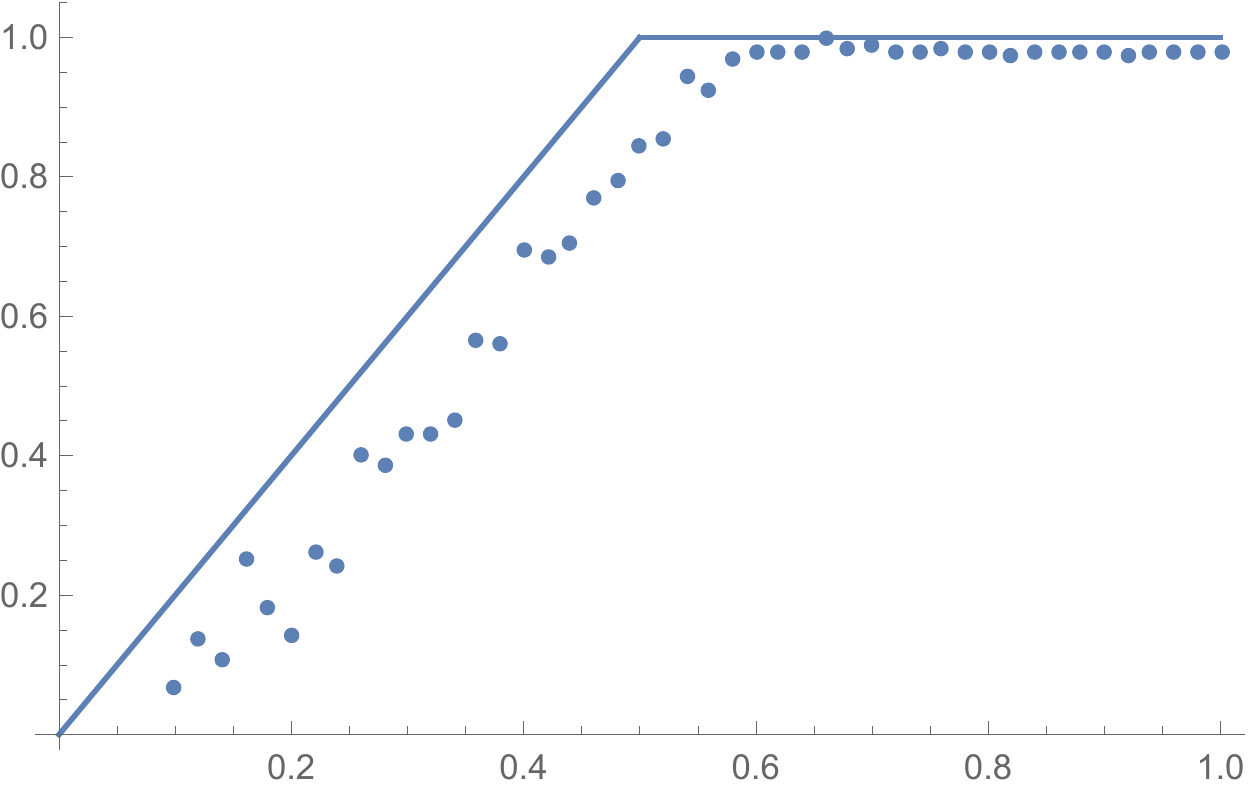}
 \caption{ A plot of the ratio  $\var(\mathcal N_{K,x})/\E(\mathcal N_{K,x})$ versus $\beta =\log K/\log N$, for $x\approx 10^8$. The smooth line is $\min(1,2\beta)$. 
  } 
 \label{fig plot var}
\end{center}
\end{figure}


One of our main goals is to justify Conjecture~\ref{sharp conj full regime}. 
In \S~\ref{sec psi} we define a suitably smoothed version of the counting function $\EN_{K,x}$ and express the corresponding variance in terms of zeros of a family of Hecke L-functions. This enables us, in \S~\ref{sec:relation to zeros}, to use GRH to give an upper bound for this variance and consequently deduce the almost-everywhere result of Theorem~\ref{thm aa sectors}. Moreover, in \S~\ref{sec:RMT} we go on to develop  a suitable random matrix theory model of this result, which gives a result corresponding to Conjecture~\ref{sharp conj full regime}. We now turn to formulating a similar problem in a function field setting, where we can prove an analogue of Conjecture~\ref{sharp conj full regime}. 

\subsection{A function field analogue}

Let $\fq$ be a finite field of cardinality $q$, from now on assumed to be odd.  
We want to write prime (irreducible monic) polynomials as 
\begin{equation}\label{P=A^2+TB^2}
P(T) = A(T)^2+TB(T)^2 
\end{equation}
with $A,B\in \fq[T]$,  
which   is equivalent to the constant term $P(0)$ being a square in $\fq$ (see e.g.\ \cite{BSSW}).  If additionally $P(0)\neq 0$, then there are exactly four such representations, obtained from \eqref{P=A^2+TB^2} by changing the signs of $A$ and $B$.  
 This decomposition gives a factorization in $\fq[T][\sqrt{-T}] = \fq[\sqrt{-T}]$ as  
$$ P =\mathfrak p \cdot \tilde {\mathfrak p}= (A+\sqrt{-T}B)(A-\sqrt{-T}B)$$
and the corresponding factorization of the  ideal $(P)\subset \fq[T]$ into a pair of conjugate prime ideals of $\fq[\sqrt{-T}]$. 
  The number $N$ of such prime polynomials $\mathfrak p(\sqrt{-T})$ of degree $\nu$  with $\mathfrak p(0)\neq 0$ satisfies  
  $$N= \frac{q^{\nu}}{\nu} +O\left(\frac{q^{\nu/2}}\nu\right)
  $$ 
  by the Prime Polynomial Theorem in $\fq[\sqrt{-T}]$.

 Denote by $S=\sqrt{-T}$ and consider the quadratic extension $\fq(T)(\sqrt{-T}) = \fq(S)$, which is still rational (genus zero). 
Let $\fq[[S]]$ be the ring of formal power series. 
It is equipped with the Galois involution 
$$
\sigma: S\mapsto -S, \quad \sigma(f)(S) = f(-S) ,
$$ 
and the norm map 
$$
\Norm:\fq[[S]]^\times \to \fq[[T]]^\times,\quad \Norm(f) = f(S) f(-S) .
$$
We denote 
$$
\Sone:=\{g\in \fq[[S]]^\times: g(0)=1, \; \Norm(g)=1\}  
$$
the formal power series with constant term $1$ and unit norm. This is a group, which is our analogue of the unit circle. It is important to note that since $q$ is odd, Hensel's Lemma tells us that the square map $u\mapsto u^2$ is an automorphism of $\Sone$, and in particular each element of $\Sone$ admits a unique square root $\sqrt{u}$.  

 We put an absolute value $|f| = q^{-\ord(f)}$  on $  \fq[[S]]$,     where $\ord(f) =\max(j: S^j\mid f)$. 
We then divide $\Sone$ into ``sectors" 
$$
\Ball(u;\kk) = \{v\in \Sone: |v-u|\leq q^{-\kk}\} .
$$
We denote by
$$
\Sone_{\kk}=\{f\in \fq[S]/(S^{\kk}): f(0)=1, \;\Norm(f):=f(-S)f(S)=1\bmod S^{\kk}\}
$$
the elements of unit norm  and constant term unity  in $\Big(\fq[S]/(S^{\kk})\Big)^\times$.   The group $\Sone_{\kk}$  parameterizes the different sectors.  
The order of $\Sone_{\kk}$ is 
\begin{equation*}
K:= \#\Sone_{\kk} = q^\kappa \;,  
\end{equation*}
where 
$$\kappa:=
\left\lfloor \frac k2 \right\rfloor ,\quad {\rm so \; that} \quad 
   \kk=\begin{cases} 2\kappa+1\\2\kappa \end{cases} .
$$

We next want to define the notion of direction (essentially an angle)  for any nonzero polynomial $f=A(T)+\sqrt{-T}B(T)\in \fq[\sqrt{-T}]$.  To motivate the definition below, recall that for a nonzero complex number $\alpha = |\alpha|e^{i\theta}$, we have $\alpha/\overline{\alpha} = e^{2i\theta}$. 
To any nonzero $f\in \fq[S]$ which is coprime to $S$, we  associate a norm-one element $U(f)\in \Sone$ via the map 
\begin{equation}\label{def of map U} 
 U:  f\mapsto\sqrt{ \frac{f}{\sigma(f)}}  \; .
\end{equation}
Note that since $f(0)\neq 0$, $f/\sigma(f)$ has constant term one, lies in $\fq[[S]]$, and has unit norm, that is  $f/\sigma(f)\in\Sone$, and hence $\sqrt{f/\sigma(f)}\in \Sone$ exists and is unique. 
Moreover, $U(cf)=U(f)$ for all scalars $c\in \fq^\times$, so that if $f\in \fq[S]$ then $U(f)$ only depends on the ideal $(f)\subset \fq[S]$ generated by $f$. 

We want to count the number of  prime  ideals $(\mathfrak p)\subset \fq[S]$ with $\mathfrak p(0)\neq 0$,  
whose directions $U(\mathfrak p)$  lie in a given sector.  
For $u\in \Sone$, let 
$$
\mathcal N_{\kk,\nu}(u):= \#\left\{(\mathfrak p) \; {\rm prime}, \;\mathfrak p(0)\neq 0: \deg \mathfrak p = \nu, \;  U(\mathfrak p)   \in \Ball(u, \kk) \right \} \;.
$$

The mean value is clearly  
$$
 \ave{\EN_{\kk,\nu}}:= \frac 1{q^\kappa}\sum_{u\in \Sone_{\kk}} \mathcal N_{\kk ,\nu}(u) 
=\frac NK\sim  \frac {q^{\nu}/\nu}{q^\kappa} \;. 
$$

For $\kk\leq \nu$ we can show (see Corollary~\ref{cor: asymp for psi}) that as $q\to \infty$,  
\begin{equation}\label{asymp for N}
 \mathcal N_{\kk,\nu}(u) = \frac NK + O \left(  q^{\nu/2} \right)    
\end{equation}
which gives an asymptotic result if $\kappa<\nu/2$. For larger values of $\kappa$, there are sectors which do not contain prime directions, as in the number field case, see Remark~\ref{a big hole}.

Our main result is the computation, in the large $q$ limit, of the number variance
$$
\Var( \mathcal N_{\kk,\nu}) := \frac 1{q^\kappa}\sum_{u\in \Sone_{\kk}} \Big| \EN_{\kk,\nu}-  \ave{\EN_{\kk,\nu}}\Big|^2 \;.
$$
\begin{theorem}\label{thm: var N pols}
Assume that $\kappa\geq 3$, or if $\kappa=2$ that $5\nmid q$. Then as $q\to \infty$,  
$$\Var( \mathcal N_{\kk,\nu}) \sim \frac{q^{\nu-\kappa}}{\nu^2}\times 
\begin{cases} 2\kappa-2,& \nu\geq 2\kappa-2 \\ \nu-1+\eta(\nu),& \kappa\leq \nu\leq 2\kappa-2
\end{cases}
$$
where $\eta(\nu)=1$ if $\nu$ is even, and $0$ otherwise. 
\end{theorem}

To compare it to our number field conjecture, here the number of sectors is 
$K=q^\kappa$, the number of directions (the number of Gaussian prime ideals $\mathfrak p$ of degree $\nu$) 
is $N\sim q^\nu/\nu $, so that the expected value is $N/K$, 
and the variance satisfies, as $q\to \infty$, 
$$
\frac{\Var (\mathcal N_{\kappa,\nu})}{N/K} \sim \begin{cases} 2\frac{\log_qK }{\log_qN}-\frac {2}{\log_qN} ,&  \log_q K\leq  \frac 12 \log_qN +1 \\ 
\\ 
1+\frac{\eta(\log_qN)-1}{\log_qN}, &
\frac 12 \log_q N + 1 \leq   \log_q K\leq  \log_q N \;.
\end{cases}  
$$
Our conjecture~\ref{sharp conj full regime}  for the number-field variance is
$$
\frac{\Var (\mathcal N_{K,N})}{N/K} \sim \min\left(1,2\frac{\log K}{\log N}\right)
$$
which is analogous to the above. 



\bigskip

\noindent{\bf Acknowledgments}  We thank  Steve Lester, for his help in the beginning of the project, and to Jon Keating, Corentin Perret-Gentil  and Peter  Sarnak for their comments. 

The research leading to these results has received funding from the
European Research Council under the European Union's Seventh
Framework Programme (FP7/2007-2013) / ERC grant agreement 
n$^{\text{o}}$ 320755.

 \section{Repulsion between angles} \label{sec:repulsion} 
 
 \subsection{Repulsion and its consequences}
 Let $\mathfrak a$ be a nonzero ideal in $\Z[i]$. If $\mathfrak a = \langle \alpha \rangle$ is generated by the Gaussian integer $\alpha$, we associate  a direction vector  $u(\mathfrak p):=u(\alpha)\in \Sone_\Q$. Since all generators of the ideal differ by multiplication by a unit $\Z[i]^\times=\{\pm 1,\pm i\}$, the direction vector $u(\mathfrak a) = e^{ i4\theta_{\mathfrak a}}$ is well-defined on ideals, while the angle   
$\theta_{\mathfrak a} $ is only defined modulo $\pi/2$.  
We can choose $\theta_{\mathfrak a}$ to lie say in $[0,\pi/2)$, corresponding to taking $\alpha=a+ib$, with $a>0$, $b\geq 0$. 
If $\mathfrak a = \langle \alpha \rangle$ for non-zero $\alpha\in \Z$,
 then $ \theta_{\mathfrak a} =0$. 


\begin{lemma}\label{lem: repulsion}
i) If $\theta_{\mathfrak a}\neq 0$  then 
$$\theta_{\mathfrak a}\gg \frac 1{\sqrt{\Norm \mathfrak a}}.
$$

ii) If   $\mathfrak p\neq \mathfrak q$   are   ideals with  distinct angles $\theta_{\mathfrak p}\neq \theta_{\mathfrak q}$ 
then 
\begin{equation*}
  |\theta_{\mathfrak p}-\theta_{\mathfrak q}|  \geq  \frac 1{\sqrt{\Norm \mathfrak p \Norm \mathfrak q}}.
\end{equation*} 
\end{lemma}
 
\begin{proof}

i) Write $\mathfrak a=\langle a+ib\rangle$ with $a,b>0$. Then 
$$\tan \theta_{\mathfrak a} = \frac ba \geq \frac 1a \geq \frac 1{\sqrt{a^2+b^2}} = \frac 1{\sqrt{\Norm \mathfrak a}}\;.
$$
Since we may assume that $\theta_{\mathfrak a}\in (0,\pi/4)$, we have $\tan \theta_{\mathfrak a}\leq \sqrt{2}\theta_{\mathfrak a} $ which gives our claim.

ii) Write 
$\mathfrak p=\langle a+ib\rangle$, $\mathfrak q = \langle c+id \rangle$,  with $a,b>0$ and $c>0$, $d\geq 0$.    
Consider the triangle having
vertices at the origin, $a+ib$ and $c+id$. Since  $\theta_{\mathfrak p}\neq \theta_{\mathfrak q}$, its
area is positive and being a lattice triangle,  its area is at
least $1/2$.

On the other hand, its area is given in terms of the angle
$\theta_{\mathfrak p}-\theta_{\mathfrak q}$ between the sides $a+ib$ and $c+id$ as
\begin{equation*}
\area=\frac 12 \sqrt{\Norm \mathfrak p}\sqrt{\Norm \mathfrak q} \sin|\theta_{\mathfrak p}-\theta_{\mathfrak q}|\;.
\end{equation*}

Thus we find
\begin{equation*}
\sqrt{\Norm \mathfrak p}\sqrt{\Norm \mathfrak q}  |\sin(\theta_{\mathfrak p}-\theta_{\mathfrak q})|\geq 1
\end{equation*}
and hence
\begin{equation*} 
|\theta_{\mathfrak p}-\theta_{\mathfrak q}|\geq \sin |\theta_{\mathfrak p}-\theta_{\mathfrak q}|
\geq \frac 1{\sqrt{\Norm \mathfrak p \Norm \mathfrak q}} \;.
\end{equation*}
\end{proof}
 
Lemma~\ref{lem: repulsion} implies that the interval $ \{ 0<\theta<1/\sqrt{x}\}$ will contain no angles $\theta_{\mathfrak p}$ for $\Norm \mathfrak p\ll x$, so  that  the number $\EN_{K,x}$ of prime angles $\theta_{\mathfrak p}$ in this interval   is zero. 
Hence we cannot expect an asymptotic formula $\EN_{K,x}\sim N/K$ to hold for {\em all}  intervals if $K\ll N^{1/2}$, while it does hold (assuming GRH) for larger intervals. Theorem~\ref{thm aa sectors}  guarantees that {\em almost all} intervals will contain angles if $K\ll N^{1-o(1)}$. 
 
\subsection{Deviations from randomness} 
 
 The existence of a ``big hole'' as above displays a striking deviation from randomness of the angles, when compared to  $N$ random angles in $[0,\pi/2)$. For these, the {\em maximal gap} is almost surely of order $\log N/N$, while Lemma~\ref{lem: repulsion}(i) guarantees a much larger gap, of size $N^{-1/2-o(1)}$.

 Another statistic which indicates that Gaussian angles behave differently  than random points is the {\em minimal spacing statistic:} For   $N$ random angles in $[0,\pi/2)$ as above, the smallest gap is almost surely of size  $\approx 1/N^2$ \cite{Levy}. In contrast, the minimal gap between the angles $\{\theta_{\mathfrak p}\neq 0:\Norm \mathfrak p\leq x\}$   is by Lemma~\ref{lem: repulsion}
$$
\min \{ |\theta_{\mathfrak p}-\theta_{\mathfrak p'}|: \Norm\mathfrak p,\Norm\mathfrak p'\leq x, \mathfrak p\neq \mathfrak p'\} \gg \frac 1{x}\approx \frac 1{N\log N},
$$
which is much bigger than the random case.

\subsection{The variance in the trivial regime}\label{Sec:trivial regime} 

We want to study fluctuations in the number $\EN_{K,x}$ of angles falling in
``random" short intervals. 
Take the interval length $1/K=o(1/x)$, equivalently the number $K$ of intervals, 
is much larger than the number $N\sim x/\log x$ of angles: $N=o(K)$. Then for almost all
intervals, we do not have any angles $\theta_{\mathfrak p}$ in the interval
$I_K(\theta)$. Nonetheless we can compute the variance in this
``trivial" regime. 
\begin{proposition}
If $x=o(K)$ then
$$ \var(\EN_{K,x})\sim \frac NK $$
\end{proposition}
\begin{proof}

We recall  definition~\eqref{def NKX}: Given an interval
$I_K(\theta)=[\theta-\frac {\pi}{4K},\theta+\frac {\pi}{4K}]$ of length
$\pi/2K$ centered at $\theta$, let\footnote{We abuse notation and use the same symbol for the interval and its indicator function.}
$$\EN_{K,x}(\theta) = \#\{\mathfrak p \;{\rm prime}, \Norm \mathfrak p\leq x:\theta_{\mathfrak p}\in I_K(\theta)\} =\sum_{\substack{\Norm \mathfrak p\leq x\\{\rm prime}}} I_K(\theta_{\mathfrak p}-\theta)$$
be the number of prime angles $\theta_{\mathfrak p}$ in $I_K(\theta)$. We will take
the center $\theta$ of the interval to be random, that is uniform in
$(0, \pi/2)$.

We compute the second moment of $\EN=\EN_{K,x}$ using its definition
\begin{equation*}
\ave{\EN^2} 
=\sum_{\Norm \mathfrak p\leq x} \sum_{\Norm \mathfrak q\leq x} 
\ave{ I_K(\theta_{\mathfrak p}-\theta )I_K(\theta_{\mathfrak q}-\theta)},    
\end{equation*}
where throughout  we use 
$$
\ave{H}:=\frac 1{\pi/2}\int_0^{\pi/2}H(\theta)d\theta .
$$
The contribution of pairs of inert primes, where $\theta_{\mathfrak p}=0$, $\mathfrak p=\langle p \rangle$,  $p=3\bmod 4$, 
$\Norm \mathfrak p = p^2\leq x$, is   
$$
\Big( \#\{  p=3\bmod 4, p\leq \sqrt{x}\}\Big)^2  \cdot \ave{  I_K(-\theta)^2} \;.
$$
Note that  $I_K^2=I_K$ and 
$$
\ave{  I_K(-\theta)^2}  =\ave{  I_K( \theta)  }=
\frac{\length(I_K)}{\pi/2} = \frac {1}{  K }\;.
$$
Moreover, the number of $p=3\bmod 4$, $p\leq \sqrt{x}$ is $\ll \sqrt{x}/\log x$.  
Hence the  contribution of pairs of inert primes is $O\Big( \frac{x}{K(\log x)^2}\Big)$.

If $\mathfrak p\neq \mathfrak q$ and at least one of $\mathfrak p$, $\mathfrak q$ is not inert, so that $\theta_{\mathfrak p}\neq \theta_{\mathfrak q}$, then Lemma~\ref{lem: repulsion} gives
$$
|\theta_{\mathfrak p}-\theta_{\mathfrak q}| \geq \frac 1x \;.
$$
For the integral $\ave{ I_K(\theta_{\mathfrak p}-\theta )I_K(\theta_{\mathfrak q}-\theta)} $ to be nonzero, it is necessary that there be some
$\theta$ so that both $\theta_{\mathfrak p},\theta_{\mathfrak q}\in I_K(\theta)$, which
forces the distance between the two angles to be at most $\pi/2K$:
\begin{equation*}
  |\theta_{\mathfrak p}-\theta_{\mathfrak q}|\leq \frac {\pi}{2K} \;.
\end{equation*}
Hence if $x=o(K)$ then such off-diagonal pairs contribute nothing.

We conclude that the second moments of $\mathcal N_{K,x}$ is essentially 
given by the sum of the diagonal terms
\begin{equation*}
\begin{split}
\ave{\EN^2} &= \sum_{ \Norm\mathfrak p\leq x}
\ave{  I_K(\theta_{\mathfrak p}-\theta)^2} 
 +O\Big( \frac{x}{K(\log x)^2}\Big)
\\& = \sum_{ \Norm\mathfrak p\leq x} \frac 1K +O\Big( \frac{x}{K(\log x)^2}\Big)
\sim  \frac NK \;.
\end{split}
\end{equation*}

We can now  compute the variance:
$$
\var(\EN) = \ave{\EN^2}-\ave{\EN}^2 \sim \frac NK - \left(\frac NK\right)^2 \;.
$$
Since $N=o(K)$ we find
$$
\var(\EN)\sim \frac NK
$$
 as claimed.
\end{proof}

\section{Almost all sectors contain an angle}\label{sec psi}

\subsection{A smooth count}
 
  Our goal in this section is to prove Theorem~\ref{thm aa sectors}, which claims (assuming GRH) 
  that in the non-trivial range   $   K\ll  X^{1-\epsilon}$, 
almost all  arcs of size $\approx 1/K$ contain at least one angle $\theta_{\mathfrak p}$, $\Norm(\mathfrak p)\leq X$. We can do so assuming GRH (for the family of Hecke L-functions).

To count the number of angles $\theta_{\mathfrak p}$ lying in a short segment of $[0,\pi/2)$, pick a window function $f\in C_c^\infty(\R)$, which we take to be even and real valued, and for $K\gg 1$ define
$$
F_K(\theta):=\sum_{j\in \Z} f\left(\frac{ K}{ \pi/2}(\theta- j \frac \pi 2 ) \right)
$$
which is $\pi/2$-periodic, and localized on a scale of $1/K$.  The Fourier expansion of $F_K$ is 
\begin{equation}\label{fourier expand FK} 
F_K(\theta)  = \sum_{k\in \Z} \^F_K(k)e^{i4 k\theta}, \qquad  \^F_K(k) = \frac 1{ K}\^f\left(\frac k{ K}\right)
\end{equation}
where the Fourier transform is  normalized as   
$\^f(y) = \int_{-\infty}^\infty f(x)e^{-2\pi iyx}dx$. Note that since $f$ is even and real valued, the same holds for $\^f$.

 Let $\Phi\in  C_c^\infty(0,\infty)$. Now set 
 $$
 \psi_{K,X}^{\rm prime}(\theta)  
:=\sum_{\mathfrak p \;{\rm prime}} \Phi \left(\frac{\Norm\mathfrak p}{X}\right) \log\Norm(\mathfrak p)    F_K(\theta_{\mathfrak p}  -\theta), 
 $$
  the sum over all prime ideals of $\Z[i]$,   
 which gives a smooth count  of prime angles $\theta_{\mathfrak p}$ lying in a smooth window defined $F_K$ around $\theta$.  
  We also define 
  $$
\psi_{K,X}(\theta)
:=\sum_{\mathfrak a  } \Phi \left(\frac{\Norm\mathfrak a}{X}\right) \Lambda(\mathfrak a) F_K(\theta_{\mathfrak a}  -\theta), 
$$ 
the sum over all powers of prime ideals, with the von Mangoldt function 
$\Lambda(\mathfrak a)=\log\Norm (\mathfrak p)$ if $\mathfrak a=\mathfrak p^r$ is a power of a prime ideal $\mathfrak p$, and equal to zero otherwise.

We next compute the mean value. 
\begin{lemma}\label{lem close means} 
The mean values of $\psi_{K,X}$ and $\psi_{K,X}^{\rm prime}$ are asymptotically 
\begin {equation}\label{main term prime}
\ave{ \psi_{K,X} }\sim \ave{\psi_{K,X}^{\rm prime}}\sim \frac{X}{K}\int_{-\infty}^\infty f(x)dx \int_0^\infty \Phi(u)du \;.
\end{equation}
Moreover, 
$$
\Big| \ave{\psi_{K,X}} - \ave{\psi_{K,X}^{\rm prime}}  \Big| \ll \frac{X^{1/2}}{K} \;.
 $$
\end{lemma} 

\begin{proof}
  The mean value is
$$
\ave{\psi_{K,X}}=
 \frac 1{K}\^f(0)\sum_{\mathfrak p \;{\rm prime} } \Phi \left(\frac{\Norm\mathfrak p}{X} \right) \Lambda(\mathfrak p)   \;.
 $$
 We can evaluate this using the Prime Ideal Theorem to obtain: 
 \begin{equation*}
\ave{\psi_{K,X}} \sim \frac{X}{K}\int_{-\infty}^\infty f(x)dx \int_0^\infty \Phi(u)du \;,
 \end{equation*}
and likewise for  $\ave{\psi_{K,X}^{\rm prime}}$. 
If in addition we use GRH, we obtain a remainder term of  $O(\frac{X^{1/2}}{K}) $ for both.

We bound the difference by  
\begin{equation*}
\begin{split}
\ave{\psi_{K,X}} - \ave{\psi_{K,X}^{\rm prime}}  &= \sum_{\mathfrak a \neq {\rm prime}}\Lambda(\mathfrak a)\Phi \left(\frac{\Norm \mathfrak a}{X} \right)\frac{\^f(0)}{K}
\\
& \ll \frac 1K 
 \sum_{\substack{ \Norm(\mathfrak a)\ll X\\ \mathfrak a \neq {\rm prime}}}\Lambda(\mathfrak a)
 \ll \frac{X^{1/2}}{K}, 
\end{split}
\end{equation*}
which shows that the mean values are close. 
\end{proof}

 Note that the inert primes $\mathfrak p=\langle p\rangle$ give angle $\theta_{\mathfrak p}=0$, but that $\Norm\mathfrak p =p^2$ so that   in $ \psi_{K,X}^{\rm prime}$, we get a contribution of size $\sqrt{X}$ if $\theta\approx 0$.  This is significantly larger than the mean value if $K\gg X^{1/2}$. 
 
 \subsection{Variance in the trivial regime}
The variance of $\psi_{K,X}^{\rm prime}$  in the trivial regime $ X=o(K)$ is:
\begin{equation}\label{trivial computation}
\Var(\psi_{K,X})\sim  \Var(\psi_{K,X}^{\rm prime})\sim  c_2(f,\Phi) \cdot  \frac{   X\log X}{K}  \;,
\end{equation}
where 
\begin{equation*}
c_2(f,\Phi):=\int_{-\infty}^\infty  f(y)^2 dy \int_0^\infty \Phi(t)^2dt  \;.
\end{equation*}
 Indeed, if $X=o(K)$ then the same argument of repulsion between angles as in \S~\ref{Sec:trivial regime} allows us to compute the second moment as asymptotically equal to the sum over the diagonal pairs
 \begin{equation*}
 \ave{|\psi_{K,X} |^2 } \sim  
   \ave{|F_K(\theta)|^2} \sum_{\mathfrak a} 
\Phi \left(\frac{\Norm(\mathfrak a)}{X} \right)^2 \Lambda(\mathfrak a)^2    \;.
\end{equation*}
By Parseval's theorem, we have
 \begin{equation*}
\begin{split}
\ave{ |F_K(\theta)|^2  } &= \frac 1{\pi/2}\int_0^{\pi/2} |F_K(\theta)|^2d\theta = \sum_{k\in \Z} |\^F_K(k)|^2  \\
&= \frac 1{K^2} \sum_{k\in \Z}  \^f\left(\frac k{K}\right)^2
\sim \frac 1{K}\int_{-\infty}^\infty f(y)^2 dy
\end{split}
\end{equation*}
and 
$$
\sum_{\mathfrak a} 
\Phi \left(\frac{\Norm(\mathfrak a)}{X} \right)^2 \Lambda(\mathfrak a)^2\sim
\int_0^\infty \Phi(t)^2dt \cdot X\log X  
$$
by the Prime Ideal Theorem. This gives the second moment as 
$$
  \ave{|\psi_{K,X}^{\rm prime}|^2 } \sim 
   \int_{-\infty}^\infty  f(y)^2 dy \int_0^\infty \Phi(t)^2dt   \cdot \frac{ X \log X}{K} \;,
$$
and since $X=o(K)$, 
we obtain \eqref{trivial computation} for $\Var(\psi_{K,X})$. The argument for 
$ \Var(\psi_{K,X}^{\rm prime})$ is identical. 

\subsection{An upper bound} 
 We  give an upper bound on the variance of $\psi_{K,X}^{\rm prime}$ in the non-trivial regime $K\ll X$, assuming GRH. 
 \begin{theorem}\label{thm upper bound on var}
Assume GRH. Then  
 \begin{equation*}
  \Var(\psi_{K,X}^{\rm prime}) \ll \frac{X}{K} (\log K)^2 \;.
 \end{equation*}
\end{theorem}
 
 From this bound we easily deduce Theorem~\ref{thm aa sectors}:  
We use Chebyshev's inequality and Theorem~\ref{thm upper bound on var} to deduce 
\begin{multline*}
\Prob\left\{\theta: |\psi_{K,X}^{\rm prime}(\theta) - \E(\psi_{K,X}^{\rm prime} )|>\frac 12 \E(\psi_{K,X}^{\rm prime})\right\}
\leq \frac{\Var(\psi_{K,X}^{\rm prime})}{\frac 14 (\E(\psi_{K,X}^{\rm prime}))^2} 
\\
\ll \frac{\frac XK (\log K)^2}{(\frac XK)^2} \ll\frac{K(\log K)^2}{X} \;.
\end{multline*}
Taking $X= K(\log K)^{2+o(1)}$  we find that for almost all $\theta$, 
$$
\psi_{K,X}^{\rm prime}(\theta) \gg \frac XK 
$$
is nonzero. Therefore the sum defining $\psi_{K,X}^{\rm prime}$ is non-empty, and since it is a sum over prime ideals giving angles $\theta_{\mathfrak p}$ in the arc of length $\approx 1/K$ around $\theta$, we find that for almost all $\theta$, such arcs contain an angle $\theta_{\mathfrak p}$ for a prime ideal with $\Norm(\mathfrak p)\leq X=K(\log K)^{2+o(1)}$. \qed

 The proof of Theorem~\ref{thm upper bound on var} will be presented in \S~\ref{sec:Proof of Theorem thm upper bound on var}. 
 
 \section{Relation to zeros of Hecke L-functions}\label{sec:relation to zeros}
 
\subsection{Hecke characters and their L-functions}\label{sec:Hecke chars} 
 The Hecke characters $\Xi_k(\alpha) = (  \alpha/\bar\alpha)^{2k}$, $k\in \Z$, give well defined functions on the ideals of $\Z[i]$. In terms of the angles associated to ideals, we have   $e^{i4k\theta_{\mathfrak p}} = \Xi_k(\mathfrak p)$.  
 
 To each such character Hecke \cite{Hecke} associated its L-function
 $$
 L(s,\Xi_k) = \sum_{0\neq \mathfrak a\subseteq \Z[i]} \frac{\Xi_k(\mathfrak a)}{(\Norm \mathfrak a)^{s}} = \prod_{\substack{\mathfrak p\\{\rm prime}}} (1-\Xi_k(\mathfrak p)(\Norm \mathfrak p)^{-s})^{-1},\quad \Re(s)>1\;.
 $$
   Note that $L(s,\Xi_k) = L(s,\Xi_{-k})$. Hecke showed that if $k\neq 0$, these functions have an analytic continuation to the entire complex plane, and satisfy  a functional equation: 
  \begin{equation}\label{functional equation}
\xi_{k}(s):=\pi^{-(s+2|k|)}\Gamma(s+2|k|) L(s,\Xi_k) =\xi_k(1-s)  \;.
\end{equation}
   
 The completed L-function $\xi_{k}(s)$ has all its zeros in the critical strip  $0<\Re(s)<1$ (the non-trivial zeros of $ L(s,\Xi_k)$), and the Generalized Riemann Hypothesis asserts that they all lie on the critical line $\Re(s)=1/2$. 
 The growth of the number of nontrivial zeros of $L(s,\Xi_k)$ in a fixed rectangle is 
\begin{equation}\label{density of zeros}
\#\{\rho: 0\leq \Im(\rho)\leq T_0\}  \sim \frac {T_0\log k}{\pi}, \quad k\to \infty, \quad T_0>0\;{\rm fixed,}
\end{equation}
in other words, the density of zeros is $\frac{\log |k|}{\pi}$.

\begin{lemma}\label{lem passing to Hecke}
\begin{equation}\label{fourier rep var psi} 
 \psi_{K,X}(\theta) = \sum_k e^{-i4k\theta} \frac 1{ K}\^f\left(\frac {k}{ K}\right) \sum_{\mathfrak a   } \Phi \left(\frac{\Norm\mathfrak a }{X} \right) \Lambda(\mathfrak a)
 \Xi_k(\mathfrak a)  
\end{equation}
 and
 \begin{equation}\label{fourier rep var} 
 \psi_{K,X}^{\rm prime}(\theta) = \sum_k e^{-i4k\theta} \frac 1{ K}\^f \left(\frac {k}{ K} \right) \sum_{\mathfrak p \;{\rm prime} } \Phi \left(\frac{\Norm\mathfrak p}{X} \right) \Lambda(\mathfrak p)
 \Xi_k(\mathfrak p) \;.
\end{equation}
\end{lemma}
\begin{proof}

  Inserting the Fourier expansion \eqref{fourier expand FK} of $F_K$ 
    gives
 $$
 \psi_{K,X}^{\rm prime}(\theta) = \sum_k e^{-i4k\theta} \frac 1{ K}\^f\left(\frac {k}{ K}\right) \sum_{\mathfrak p} \Phi \left(\frac{\Norm\mathfrak p}{X} \right) \Lambda(\mathfrak p) e^{i4k\theta_{\mathfrak p}} \;.
  $$
 Now note that $e^{i4k\theta_{\mathfrak p}} = \Xi_k(\mathfrak p)$ is the Hecke character, to obtain \eqref{fourier rep var}. The same argument gives \eqref{fourier rep var psi}. 
  \end{proof}


The zero mode $k=0$ in \eqref{fourier rep var}  is the 
mean value \eqref{main term prime}.  
The same holds for $\psi_{K,X}$. 

        \subsection{An Explicit Formula}

  \begin{proposition}\label{Explicit Formula}
   Let $\Phi\in C_c^\infty(0,\infty)$, and   
 $$\tilde \Phi(s) =   \int_0^\infty \Phi(x) x^s\frac{dx}{x}$$
be its Mellin transform.  Then  for $k\neq 0$ and $X\gg_\Phi 1$, 
\begin{multline*}
\sum_{\mathfrak a} \Lambda(\mathfrak a) \Xi_k(\mathfrak a)\Phi \left(\frac{\Norm(\mathfrak a)}{X} \right) =- \sum_{\xi_{k}(\rho)=0} \tilde \Phi(\rho)X^\rho 
\\+\frac 1{2\pi i}\int_{(2)} \left\{   \frac{\Gamma'}{\Gamma}(s+2|k|) + \frac{\Gamma'}{\Gamma}(1-s+2|k|) \right\} \tilde \Phi(s)X^s ds  , 
\end{multline*}
where the sum on the RHS is over all non-trivial zeros of  $L(s,\Xi_k)$.
  \end{proposition}
 \begin{proof}
 We abbreviate $L_k(s):=L(s,\Xi_k)$. 
Using Mellin inversion $ \Phi(x) = \frac 1{2\pi i} \int_{\Re(s)=2} \tilde \Phi(s)x^{-s}ds $ we obtain
 \begin{equation*}
 \begin{split}
 \sum_{\mathfrak a} \Lambda(\mathfrak a) \Xi_k(\mathfrak a)\Phi \left(\frac{\Norm(\mathfrak a)}{X}\right) &= \frac 1{2\pi i}\int_{(2)}  \sum_{\mathfrak a} \Lambda(\mathfrak a) \Xi_k(\mathfrak a) \frac{X^s}{\Norm(\mathfrak a)^s} \tilde \Phi(s) ds
 \\
&=  \frac 1{2\pi i}\int_{(2)} -\frac{L_{k}'}{L_{k}}(s) \tilde \Phi(s) X^sds \;.
\end{split}
 \end{equation*}
In terms of the completed L-function $\xi_{k}(s)$, the logarithmic derivative of $L(s,\Xi_k)$  is 
 $$
  -\frac{L_{k}'}{L_{k}}(s)  = -\log \pi  + \frac{\Gamma'}{\Gamma}(s+2|k|)  -\frac{\xi_{k}'}{\xi_{k}}(s) \;.
  $$
Inserting into the above gives 
 \begin{equation*}
 \begin{split}
\frac 1{2\pi i}\int_{(2)} -\frac{L_{k}'}{L_{k}}(s) \tilde \Phi(s) X^sds & =
 \frac 1{2\pi i}\int_{(2)}\Big (  -\log \pi  + \frac{\Gamma'}{\Gamma}(s+2|k|) \Big) \tilde \Phi(s) X^sds  \\
 &+ 
 \frac 1{2\pi i}\int_{(2)} -\frac{\xi_{k}'}{\xi_{k}}(s) \tilde \Phi(s) X^sds \;.
 \end{split}
 \end{equation*}

We shift the contour in the integral to $\Re(s)=-1$, picking up the poles of $-\frac{\xi_{k}'}{\xi_{k}}(s)$, which are all simple poles with residue $-1$ at the non-trivial zeros of $L_{k}(s)$, giving 
$$
 \frac 1{2\pi i}\int_{(2)} -\frac{\xi_{k}'}{\xi_{k}}(s) \tilde \Phi(s) X^sds 
  =- \sum_\rho \tilde \Phi(\rho)X^\rho + \frac 1{2\pi i}\int_{(-1)}  -\frac{\xi_{k}'}{\xi_{k}}(s) \tilde \Phi(s) X^sds \;.
  $$
Changing variables $s\mapsto 1-s$ gives
$$
 \frac 1{2\pi i}\int_{(-1)}  -\frac{\xi_{k}'}{\xi_{k}}(s) \tilde \Phi(s) X^sds =
  \frac 1{2\pi i}\int_{(2)}  -\frac{\xi_{k}'}{\xi_{k}}(1-s) \tilde \Phi(1-s) X^{1-s}ds \;.
  $$
  
      The functional equation \eqref{functional equation}  of $L(s,\Xi_k)$ 
implies 
 $$
  -\frac{\xi_{k}'}{\xi_{k}}(s)  =  \frac{\xi_{k}'}{\xi_{k}}(1-s) 
  $$
  which gives 
  $$
  \frac 1{2\pi i}\int_{(2)}  -\frac{\xi_{k}'}{\xi_{k}}(1-s) \tilde \Phi(1-s) X^{1-s}ds =
  \frac 1{2\pi i}\int_{(2)}  \frac{\xi_{k}'}{\xi_{k}}( s) \tilde \Phi(1-s) X^{1-s}ds \;.
  $$
Returning to the incomplete L-function gives 
\begin{equation*}
\begin{split}
  \frac 1{2\pi i}\int_{(2)}&
    \frac{\xi_{k}'}{\xi_{k}}( s) \tilde \Phi(1-s) X^{1-s}ds
  \\ =
  & \frac 1{2\pi i}\int_{(2)}  \left(  -\log \pi +\frac{\Gamma'}{\Gamma}(s+2|k|) 
    +   \frac{L_{k}'}{L_{k}}( s)\right) \tilde \Phi(1-s) X^{1-s}ds \\
=&-\log \pi  \frac 1{2\pi i}\int_{(2)}  \tilde \Phi( s) X^{ s}ds  +
  \frac 1{2\pi i}\int_{(2)} \frac{\Gamma'}{\Gamma}(1-s+2|k|) \tilde \Phi( s) X^{ s}ds 
  \\
  & +\frac 1{2\pi i}\int_{(2)}  \frac{L_{k}'}{L_{k}}( s)  \tilde \Phi(1-s) X^{1-s}ds \;.
\end{split}
\end{equation*}

 By Mellin inversion,
 $$ \frac 1{2\pi i}\int_{(2)}  \tilde \Phi( s) X^{ s}ds  = \Phi \left(\frac 1X \right),
$$
 which vanishes for $X\gg1$ as $\Phi$ is compactly supported in $(0,\infty)$.   
 Likewise, 
\begin{multline*}
 \frac 1{2\pi i}\int_{(2)}  \frac{L_{k}'}{L_{k}}( s)  \tilde \Phi(1-s) X^{1-s}ds  = 
 -\frac 1{2\pi i}\int_{(2)}  \sum_{\mathfrak a} \frac{\Lambda(\mathfrak a) \Xi_k(\mathfrak a)}{\Norm(\mathfrak a)^s}  X^{1-s}  \tilde \Phi(1-s) ds
 \\ 
 =-  \sum_{\mathfrak a} \frac{\Lambda(\mathfrak a) \Xi_k(\mathfrak a)}{\Norm(\mathfrak a) }\frac 1{2\pi i}\int_{(2)}    \tilde \Phi(1-s)  (X \Norm(\mathfrak a))^{1-s}ds
 \\ =
 -  \sum_{\mathfrak a} \frac{\Lambda(\mathfrak a) \Xi_k(\mathfrak a)}{\Norm(\mathfrak a) } \Phi \left(\frac 1{X\Norm(\mathfrak a)} \right) =0, 
\end{multline*}
since  each term  vanishes for $X\gg 1$ (independently of $\mathfrak a$, since $\Norm(\mathfrak a)\geq 1$). 

Collecting terms, we find 
 \begin{multline*}
 \sum_{\mathfrak a} \Lambda(\mathfrak a) \Xi_k(\mathfrak a)\Phi \left(\frac{\Norm(\mathfrak a)}{X} \right)   = - \sum_\rho \tilde \Phi(\rho)X^\rho\\
+
  \frac 1{2\pi i}\int_{(2)} \left\{ \frac{\Gamma'}{\Gamma}( s+2|k|)  +  \frac{\Gamma'}{\Gamma}(1-s+2|k|)
  \right\}\tilde \Phi( s) X^{ s}ds 
\end{multline*}
as claimed.
  \end{proof}

\begin{lemma}\label{lem:integrate by parts}
For $k\neq 0$, 
$$ \frac 1{2\pi i}\int_{(2)} \left\{ \frac{\Gamma'}{\Gamma}( s+2|k|)  +  \frac{\Gamma'}{\Gamma}(1-s+2|k|)
  \right\}\tilde \Phi( s) X^{ s}ds  \ll \frac{X^{1/2}\log 2|k|}{(\log X)^{100}}\;.
$$
\end{lemma}
\begin{proof}
Note that the integrand is analytic in $-2<\Re(s)<3$, so we may shift the contour of integration to $\Re(s)=1/2$. Let 
$$ h_k(t):= \left\{ \frac{\Gamma'}{\Gamma}\left(\tfrac 12+it+2|k|\right)  +  \frac{\Gamma'}{\Gamma}\left(\tfrac 12-it+2|k|\right)  \right\} \tilde \Phi\left( \tfrac 12 +it\right)\;.
$$
The integral is essentially $X^{1/2}$ times the Fourier transform $ \widehat h_k(\log X)$, that is 
$$
X^{1/2}\frac 1{2\pi} \int_{-\infty}^\infty h_k(t)e^{ i t\log X}dt\;.
$$
We can estimate the derivatives of $h_k(t)$ by using Stirling's formula and the rapid decay of $\tilde \Phi(\frac 12 +it)$ as being bounded by 
$$ |h_k^{(j)}(t) |\ll \frac{\log 2|k|}{(1+|t|)^{200}}\;.
$$
Hence integration by parts shows that the Fourier transform of $h_k$ is bounded by 
$$ | \widehat h_k(\log X)  | \ll \frac{\log 2|k|}{(\log X)^{100}}, 
$$
which proves the Lemma. 
\end{proof}

 From Lemma~\ref{lem passing to Hecke},   Proposition~\ref{Explicit Formula} and Lemma~\ref{lem:integrate by parts} we deduce: 
 \begin{corollary}\label{cor form for psi}
 Assume GRH. Then 
\begin{multline*}
 \psi_{K,X}(\theta)-\ave{  \psi_{K,X} } =
\\
-  X^{1/2}  \sum_{k\neq 0 } e^{-i4k\theta} \frac 1{ K}\^f\left(\frac {k}{ K}\right) 
\left(  \sum_{\xi_{k}(\frac 12 + i\gamma_{k,n})=0} \tilde \Phi\left(\frac 12 + i\gamma_{k,n}\right)X^{i\gamma_{k,n}}  
  +O\Big(\frac{ \log K}{(\log X)^{100}}\Big) \right) \;.
\end{multline*}
 \end{corollary}
 

 Averaging Corollary~\ref{cor form for psi} over $\theta$ we find
 \begin{corollary}\label{cor form for var} 
 Assume GRH. Then 
\begin{multline*}
  \Var(\psi_{K,X}) =\\
    \frac{X}{  K^2}  
  \sum_{k\neq 0} \^f\left(\frac {k}{ K}\right)^2 \left(  \sum_{\xi_{k}(\frac 12 + i\gamma_{k,n})=0} \tilde \Phi \left(\frac 12 + i\gamma_{k,n}\right)X^{i\gamma_{k,n}}  
  +O\Big(\frac{ \log K}{(\log X)^{100}}\Big) \right)^2 \;.
\end{multline*}
 \end{corollary}


 \begin{corollary}\label{cor: bd for var psi}
 Assume GRH. Then
 $$
\Var(\psi_{K,X}) \ll  \frac XK (\log K)^2 \;,
 $$
 \end{corollary}
 \begin{proof}
  We use GRH to obtain  $|X^{i\gamma_{k,n}}|=1$ so that 
\begin{equation}\label{using GRH}
  \left|  \sum_n \tilde\Phi \left(\frac 12 +i\gamma_{k,n}\right) X^{i\gamma_{k,n}}\right| \leq 
     \sum_n |\tilde\Phi \left(\frac 12 +i\gamma_{k,n}\right)| \;. 
 \end{equation}
   We use a standard bound for the number of zeros of $L(s,\Xi_k)$ in an interval (see \cite[Proposition 5.7]{IK}):
\begin{equation}\label{bd for number of zeros} 
 \#\{n: \Im(\rho_{n,k})\in[T-1,T+1]\} \ll  \log (|\frac 12 +iT|+2|k|) \;.
\end{equation}

Note that  $ \tilde\Phi$ decays rapidly  in vertical strips, say 
$$
| \tilde\Phi\left(\frac 12 +iu \right) | \ll_\Phi \frac 1{(1+|u|)^{100}}, 
$$
which together with  \eqref{bd for number of zeros}  gives 
\begin{equation}\label{bd for sum over zeros}
\begin{split}
| \sum_n \tilde\Phi\left(\frac 12 +i\gamma_{k,n}\right) | &\leq \sum_{j\in \Z} \sum_{n:j\leq \gamma_{k,n}<j+1} | \tilde\Phi\left(\frac 12 +i \gamma_{k,n}\right) | \\
&\ll_\Phi \sum_{j\in \Z} \frac 1{(1+|j|)^{100}} \log(|2k|+|j|) \ll \log (2|k|)\;.
\end{split}
\end{equation}

Inserting \eqref{bd for sum over zeros}  into Corollary~\ref{cor form for var}   gives  
$$
\Var(\psi_{K,X}) \ll  \frac{  X}{ K^2}\sum_{k> 0}
 |\^f\left(\frac {k}{ K}\right)|^2 (\log 2k)^2 \ll  \frac XK (\log K)^2 \;,
 $$
as claimed. 
\end{proof}

\subsection{Primes vs prime powers}
We pass from a sum over prime ideals to a sum over all prime powers:
 \begin{lemma}\label{lem primes vs nonprimes}
Assume GRH. For $k\neq 0$ such that $\log |k|\ll \log  X$,
\begin{equation*}
   \sum_{\mathfrak a} \Lambda(\mathfrak a) \Xi_k(\mathfrak a)\Phi\left(\frac{\Norm(\mathfrak a)}{X}\right)  = \sum_{\mathfrak p\;{\rm prime}} \Lambda(\mathfrak p) \Xi_k(\mathfrak p)\Phi\left(\frac{\Norm(\mathfrak p)}{X}\right)
  +O\Big( X^{1/3}\Big)\;.
  \end{equation*}
   \end{lemma}
 \begin{proof}
 
 We denote 
 $$
 \Sigma_{\rm prime}(X,k,\Phi):= \sum_{\mathfrak p\;{\rm prime}} \Lambda(\mathfrak p) \Xi_k(\mathfrak p)\Phi\left(\frac{\Norm(\mathfrak p)}{X}\right)
 $$
 and
  $$
 \Sigma_{\rm all}(X,k,\Phi): =\sum_{\mathfrak a} \Lambda(\mathfrak a) \Xi_k(\mathfrak a)\Phi\left(\frac{\Norm(\mathfrak a)}{X}\right) \;.
 $$
 
 Assuming GRH, we have 
  \begin{equation*}
   \Sigma_{\rm all}(X,k,\Phi) \ll X^{1/2}\log(2|k|)  \;. 
   \end{equation*}
   Indeed,  from the Explicit Formula (Proposition~\ref{Explicit Formula}), Lemma~\ref{lem:integrate by parts} and GRH we have 
\begin{equation*}
\begin{split}
  \Sigma_{\rm all}(X,k,\Phi)& =- \sum_{\xi_k(\frac 12 +i\gamma)=0} \tilde \Phi\left(\frac 12 +i\gamma\right)X^{\frac 12 +i\gamma} 
\\&+\frac 1{2\pi i}\int_{(2)} \left\{   \frac{\Gamma'}{\Gamma}(s+2k) + \frac{\Gamma'}{\Gamma}(1-s+2k) \right\} \tilde \Phi(s)X^s ds  
\\
&
\ll X^{1/2}  \sum_{\xi_{k}(\frac 12 +i\gamma)=0}  |\tilde \Phi\left(\frac 12 +i\gamma\right)| +
 \frac{X^{1/2}\log 2|k|}{(\log X)^{100}} \ll X^{1/2}\log(2|k|)
 \end{split}
 \end{equation*}
   on using the density of zeros of $L(s,\Xi_k)$ \eqref{density of zeros}.
  
  Next we crudely bound the contribution $\Sigma_{\geq 2}(X,k,\Phi)$ to $ \Sigma_{\rm all}(X,k,\Phi)$ of the higher prime powers $\mathfrak p^j$, $j\geq 2$:
 \begin{equation*}
\begin{split}
 \Sigma_{ \geq 2}(X,k,\Phi)&:= \sum_{\mathfrak p\;{\rm prime}} \sum_{j\geq 2}\Lambda(\mathfrak p^j) \Xi_k(\mathfrak p^j)\Phi\left(\frac{\Norm(\mathfrak p^j)}{X}\right)
 \\ &\leq   \sum_{\mathfrak p\;{\rm prime}} \log \Norm(\mathfrak p) \sum_{j\geq 2} \Phi\left(\frac{\Norm(\mathfrak p)^j}{X}\right)
  \\
  &\ll  \sum_{\substack{\mathfrak p\;{\rm prime}\\ \Norm(\mathfrak p)\ll X^{1/2}}} \log \Norm(\mathfrak p) \frac{\log X}{\log \Norm(\mathfrak p)}
\\
&  \ll X^{1/2} \;.
 \end{split}
 \end{equation*}
  Therefore we obtain a crude a priori bound on the contribution of primes:
 \begin{equation}\label{a priori primes}
 \Sigma_{\rm prime}(X,k,\Phi)=  \Sigma_{\rm all}(X,k,\Phi)-  \Sigma_{ \geq 2}(X,k,\Phi) \ll  X^{1/2}\log(2|k|)\;.
  \end{equation} 
    
  We now seek a more refined estimate.   
In the sum $  \Sigma_{\rm all}(X,k,\Phi)$ over all prime power, we separately treat the contributions   of primes, of squares of primes, and of higher powers:
$$
\Sigma_{\rm all}(X,k,\Phi) =  \Sigma_{\rm prime}(X,k,\Phi) + \Sigma_2(X,k,\Phi)  + \Sigma_{\geq 3}(X,k,\Phi)
$$
where 
$$
\Sigma_{ \geq 3}(X,k,\Phi) := \sum_{\mathfrak p\;{\rm prime}} \sum_{j\geq 3}\Lambda(\mathfrak p^j) \Xi_k(\mathfrak p^j)\Phi\left(\frac{\Norm(\mathfrak p^j)}{X}\right)
$$
and 
 \begin{equation*}
\begin{split}
\Sigma_2(X,k,\Phi) &=  \sum_{\mathfrak p\;{\rm prime}} \Lambda(\mathfrak p^2) \Xi_k(\mathfrak p^2)\Phi\left(\frac{\Norm(\mathfrak p^2)}{X}\right) \\
&=  \sum_{\mathfrak p\;{\rm prime}} \log \Norm(\mathfrak p)\Xi_{2k}(\mathfrak p)  \Phi\left(\frac{\Norm(\mathfrak p )^2}{X} \right) \;.
  \end{split}
 \end{equation*}
 By definition,
 $$
\Sigma_2(X,k,\Phi)  =  \Sigma_{\rm prime}(X^{1/2},2k,\Phi_2) 
  $$
 where  $\Phi_2(u) = \Phi(u^2)$. Therefore inputting the a priori bound \eqref{a priori primes} (which uses GRH to get cancellation) gives
 \begin{equation*}
 \Sigma_2(X,k,\Phi)\ll X^{1/4} \log(2|k|)\;.
 \end{equation*}
 
 For the contribution of higher powers, we use 
 \begin{equation*}
\begin{split}
 \Sigma_{ \geq 3}(X,k,\Phi) & \ll 
   \sum_{\mathfrak p\;{\rm prime}} \log \Norm(\mathfrak p) \sum_{j\geq 3} \Phi\left(\frac{\Norm(\mathfrak p)^j}{X}\right)
\\&   \ll  \sum_{\substack{\mathfrak p\;{\rm prime}\\ \Norm(\mathfrak p)\ll X^{1/3}}} \log \Norm(\mathfrak p) \frac{\log X}{\log \Norm(\mathfrak p)}
\\
&  \ll X^{1/3}\;.
\end{split}
 \end{equation*}
 Thus we obtain 
 $$
 \Sigma_{\rm all}(X,k,\Phi) =  \Sigma_{\rm prime}(X,k,\Phi)   + O\left(X^{1/4}\log(2|k|)\right)  + O\left(X^{1/3}\right)  ,
 $$
  which gives us the result since $\log |k|\ll \log X$. 
 \end{proof}

\begin{lemma}\label{lem: compare psi and psiprime}
 Assume GRH. Then
 $$\ave{| \psi_{K,X}  -  \psi_{K,X}^{\rm prime}|^2} \ll \frac{X^{2/3}}K . 
 $$
 \end{lemma} 
 \begin{proof}
 We use Lemma~\ref{lem passing to Hecke} to write 
$$
 \psi_{K,X}(\theta)  -  \psi_{K,X}^{\rm prime}(\theta)  = \frac  1K \sum_{k} e^{-i4k\theta} \^f\left(\frac kK\right) \sum_{\mathfrak a \neq {\rm prime}}\Lambda(\mathfrak a)\Phi\left(\frac{\Norm \mathfrak a}{X}\right)\Xi_k(\mathfrak a) \;.
$$
The term $k=0$ is the difference between mean values, which by Lemma~\ref{lem close means}  is $O(  X^{1/2}/K)$. 
Hence 
\begin{equation*}
\begin{split}
 \psi_{K,X}(\theta)  -  \psi_{K,X}^{\rm prime}(\theta) 
  &= \frac  1K \sum_{k\neq 0} e^{-i4k\theta} \^f\left(\frac kK\right) 
 \sum_{\mathfrak a \neq {\rm prime}}\Lambda(\mathfrak a)
 \Phi\left(\frac{\Norm \mathfrak a}{X}\right) \Xi_k(\mathfrak a) 
\\
&\qquad + O\left(\frac{X^{1/2}}{K}\right)\\
& =I+ O\left(\frac{X^{1/2}}{K}\right)
\end{split}
\end{equation*}
 say. Hence it suffices to show that $\ave{I^2}\ll X^{2/3}/K$. 
 
 We have 
 $$
 \ave{I^2} = \frac 1{K^2}\sum_{k\neq 0} \^f\left(\frac kK\right)^2  
 \Big| \sum_{\mathfrak a \neq {\rm prime}}\Lambda(\mathfrak a)\Phi\left(\frac{\Norm \mathfrak a}{X}\right)\Xi_k(\mathfrak a) \Big|^2 \;.
 $$
 By Lemma~\ref{lem primes vs nonprimes}, the sum over $\mathfrak a$ non prime is $O(X^{1/3})$ (assuming $\log K\ll \log X$), and therefore
 $$
  \ave{I^2} \ll  \frac 1{K^2}\sum_{k\neq 0 }  \^f\left(\frac kK\right)^2  X^{2/3} \ll \frac{X^{2/3}}K
  $$
 as desired.  
 \end{proof}

\subsection{Proof of Theorem~\ref{thm upper bound on var}}\label{sec:Proof of Theorem thm upper bound on var}

 We want to show that 
 $$  \var(\psi_{K,X}^{\rm prime}) =|| \psi_{K,X}^{\rm prime}-\ave{\psi_{K,X}^{\rm prime}}||_2^2
 \ll \frac{X}{K}(\log K)^2$$
 where 
 $$||f||_2^2 =\frac 1{\pi/2}\int_0^{\pi/2}|f(\theta)|^2 d\theta
 $$
 is the standard $L^2$ norm on $[0,\pi/2]$. 
 
 Using the triangle inequality, we have
\begin{multline*}
 || \psi_{K,X}^{\rm prime}-\ave{\psi_{K,X}^{\rm prime}}||_2\leq 
 || \psi_{K,X}^{\rm prime}-\psi_{K,X}||_2 +|| \psi_{K,X} -\ave{\psi_{K,X} }||_2
 \\
 +|\ave{\psi_{K,X}}-\ave{\psi_{K,X}^{\rm prime}}| \;.
 \end{multline*}

By Lemma~\ref{lem: compare psi and psiprime} 
 $$ || \psi_{K,X}^{\rm prime}-\psi_{K,X}||_2 =  \ave{| \psi_{K,X}  -  \psi_{K,X}^{\rm prime}|^2}^{1/2}\ll \Big(\frac{X^{2/3}}K\Big)^{1/2}\;;
 $$
  by Corollary~\ref{cor: bd for var psi}, 
$$
|| \psi_{K,X} -\ave{\psi_{K,X} }||_2 = \Big(\Var(\psi_{K,X})\Big)^{1/2}\ll \Big(\frac XK(\log K)^2\Big)^{1/2} \;,
$$
and by Lemma~\ref{lem close means},  the mean values are close:
$$
\Big| \ave{\psi_{K,X}} - \ave{\psi_{K,X}^{\rm prime}} \Big| \ll \frac{X^{1/2}}{K}\;.
$$
Thus we obtain 
\begin{equation*}
\begin{split}
 || \psi_{K,X}^{\rm prime}-\ave{\psi_{K,X}^{\rm prime}}||_2&\ll
 \Big(\frac{X^{2/3}}K\Big)^{1/2} + \Big(\frac XK(\log K)^2\Big)^{1/2} + \frac{X^{1/2}}{K}
 \\
 &\ll \Big(\frac XK(\log K)^2\Big)^{1/2} , 
 \end{split} 
 \end{equation*}
   hence  
 $$
 \Var(\psi_{K,X}^{\rm prime})  \ll  \frac XK (\log K)^2
 $$ 
 which proves Theorem~\ref{thm upper bound on var}. \qed


\section{A random matrix theory model}\label{sec:RMT}
 
 In this section we present a conjecture for the variance  of the smooth count 
 $\psi_{K,X}$: 
 \begin{conjecture}\label{conj full regime}
\begin{equation*}
 \Var( \psi_{K,X})\sim   c_2(f,\Phi) \frac {X   }{K} \cdot  \min\left(\log X, 2 \log K \right) 
 \end{equation*}
 where 
 \begin{equation*}
   c_2(f,\Phi) =  \int_{-\infty}^\infty  f(y)^2dy  \int_0^\infty \Phi(t)^2dt  \;.
 \end{equation*}
 \end{conjecture}
 
 Note that  Conjecture~\ref{conj full regime} coincides with our result  \eqref{trivial computation}    in the trivial regime  range $K\gg X$.

 To recover Conjecture~\ref{sharp conj full regime} from Conjecture~\ref{conj full regime}, we can (at a heuristic level) pass to an actual count with sharp cutoffs: Take $f=\mathbf 1_{[-1/2,1/2]}$ and $\Phi = \mathbf 1_{(0,1]}$, and replace the weight $\Lambda(\mathfrak p)$   by $\log X$ throughout, and ignore the contribution of higher powers of primes.

We use Corollary~\ref{cor form for var} with  $X=K^\alpha$ for $\alpha>0$, and note that since  
 $\^f$  is even, and  $\xi_{-k}(s) = \xi_{k}(s)$, we can pass to a sum over positive $k$'s, to obtain  
\begin{equation}\label{Var as average form factor} 
 \Var(\psi_{K,X})\sim \frac{2  X}{ K^2}\sum_{k> 0}
 \^f\left(\frac {k}{K}\right)^2 \Big|  \sum_j \tilde \Phi\left(\frac 12 +i\gamma_{k,j}\right) 
 e^{ i\alpha \log K \gamma_{k,j}} \Big|^2  \;,
 \end{equation}
 the inner sums over all non-trivial zeros of $L(s,\Xi_{k})$; we have ignored the remainder term in Corollary~\ref{cor form for var} as it can be seen to be $o(X/K)$ by using \eqref{bd for sum over zeros}. 

Let 
\begin{equation}\label{choose r I}
 n :=\frac \alpha 2 \frac{\log K}{\pi} \;,
 \end{equation}
and
$$
\mathcal S_n(\Xi_{k})=\sum_j \tilde \Phi\left(\frac 12 + i\gamma_{k,j}\right) e^{2\pi i n\gamma_{k,j}}\;.
$$
Since the density of zeros of $L(s,\Xi_{k})$ is about $\approx \log |k|$, the   sum in $\mathcal S_n(\Xi_{k})$ is over $O(\log K)$ zeros. 

  Conjecture~\ref{conj full regime}  is clearly implied by 
 \begin{conjecture}\label{conj for bigsum}
 Fix $\alpha>0$. Then as $K\to \infty$, 
 \begin{equation}\label{bigsum}
\frac 2K \sum_{k> 0}  \^f\left(\frac{k}{K}\right)^2 \Big| \mathcal S_n(\Xi_{k}) \Big|^2 
\sim   c_2(f,\Phi) 
 \log K \min(  \alpha, 2)  \;.
\end{equation}
\end{conjecture}


\subsection{The model} 
We model the  sum $\mathcal S_n(\Xi_{k})$ by replacing the zeros  of  $L(s,\Xi_{k})$ by the eigenvalues of a fictitious $N\times N$ (diagonal) unitary matrix 
$$ 
U = {\rm diag} (e^{2\pi i\gamma_j})_{j=1,\dots, N} .
$$
We may want to require that $U$ be symplectic\footnote{or orthogonal}, in which case $N=2g$ is even and the eigenphases $\gamma_j$ will come in conjugate pairs $\gamma_{N-j} = -\gamma_j$, $j=1,\dots, g$. 

We choose $N$  so that the density of angles, namely $N$, 
matches the density of zeros of $L(s,\Xi_{k})$ 
by requiring  
 \begin{equation}\label{matching N and K}
 N \approx \frac{\log K}{\pi} \;. 
\end{equation} 
 
We replace $\tilde \Phi(\frac 12 + i\gamma)$ by a  periodic  function $w(\gamma) = w(\gamma+1)$, to get a linear statistic
$$
S_n(U):=\sum_{j=1}^N w(\gamma_j) e^{2\pi i n\gamma_j}  \;.
$$
Expanding $ w(\gamma) = \sum_{\ell\in \Z} \^w(\ell) e^{2\pi i \ell\gamma}$ 
in a Fourier series we obtain 
\begin{equation}\label{expand W}
S_n(U) =   \sum_\ell \^w(\ell) \sum_j e^{2\pi i (n+\ell)\gamma_j}=\sum_m \^w(m-n) \tr   (U^{m})   \;.
\end{equation}

We obtain  the following  model for the sum \eqref{bigsum}: 
$$
\eqref{bigsum}\quad  \longleftrightarrow \quad  \frac 2K\sum_{k>0}  \^f\left(\frac {k}{K}\right)^2 \Big| S_n(U_k)  \Big|^2 , 
$$
where the unitary matrices $U_k$ are picked uniformly and independently from a certain subgroup $G(N)\subseteq U(N)$ of unitary $N\times N$ matrices, $N\approx \frac 1\pi \log K$, say $G(N) = U(N)$ is the full unitary group, or the symplectic group  $G(N) = {\rm USp}(N)$ (possible only when $N$ is even).

We now replace the discrete average $\frac 2K\sum_{k>0}  \^f\left(\frac {k}{K}\right)^2 H(U_k)$ by the continuous average $c_f \int_{G(N)} H(U)dU$ with respect to the Haar probability measure on $G(N)$, with $c_f$ chosen so that the two averages coincide when the test function $H(U)\equiv 1$ is constant, that is 
$$
c_f:=\lim_{K\to \infty} \frac 2K\sum_{ k>0 }
\^f\left(\frac{k}{K}\right)^2 = \int_{-\infty}^\infty f(y)^2dy
$$
 (recalling that $f$ is even and real valued).  
  Therefore we  model \eqref{bigsum} by the matrix integral 
\begin{equation}
\eqref{bigsum}\quad  \longleftrightarrow c_f \int_{G(N)} |S_n(U)|^2 dU\;,
\end{equation}
where $n\approx N$ grows linearly with the matrix size $N$, precisely so that under the correspondence \eqref{matching N and K} and \eqref{choose r I}, 
$
n  \longleftrightarrow \frac{ \alpha}{2} \frac{\log K}{\pi} 
$  
is assumed to be an integer.

We claim that for all the classical groups ($G =$ U,  USp, O) under these conditions the answer is 
\begin{proposition}\label{RMT integral}
For $G =$ $\rm U$, $\rm  USp$, $\rm O$, and $n\approx N$, as $N\to \infty$
 \begin{equation*}
\int_{G(N)} |S_n(U)|^2 dU \sim \min(n,N)\int_0^1|w(\gamma)|^2d\gamma \;.
\end{equation*}
\end{proposition}

Therefore we are led to conjecture~\ref{conj for bigsum}, once we understand the analogue of $\int_0^1|w(\gamma)|^2d\gamma$: Recall that $w(\gamma)$ corresponded to $\tilde \Phi(\frac 12 +i\gamma)$, which we can write in terms of $\phi(t):=\Phi(e^t)e^{t/2}$ as 
$$ \tilde \Phi\left(\frac 12 +i\gamma\right) = \int_0^\infty \Phi(x)x^{\frac 12+i\gamma} \frac{dx}{x} = \int_{-\infty}^\infty \Phi\left(e^y\right)e^{y/2} e^{i\gamma y}dy = \^\phi\left(-\frac{\gamma}{2\pi}\right) \;.
$$
Hence 
$\int_0^1|w(\gamma)|^2d\gamma$ corresponds to 
$$
\int_{-\infty}^\infty \^\phi\left(-\frac{\gamma}{2\pi}\right)^2 d\gamma = 2\pi \int_{-\infty}^\infty \phi(t)^2dt = 2\pi \int_0^\infty \Phi(x)^2dx \;. 
$$  
Thus we obtain Conjecture~\ref{conj for bigsum}
$$\eqref{bigsum} \sim 
 c_f  2\pi \int_0^\infty \Phi(x)^2dx \cdot \frac{\log K}{\pi} \min\left(\frac \alpha 2, 1\right)  =
 c_2(f,\Phi)\cdot \log K \min(  \alpha , 2)  \;.
$$

\subsection{Proof of Proposition~\ref{RMT integral}}
\begin{proof}

We use the Fourier expansion \eqref{expand W} to obtain
\begin{equation*}
\int_{G(N)} |S_n(U)|^2 dU= \sum_{m,m'} \^w(m-n) \overline{\^w(m'-n)} \int_{G(N)} \tr (U^m) \overline {\tr(U^{m'}) }dU \;.
\end{equation*}
 We trivially have $|\tr U^m|\leq N$, and since $n\approx  N$ and $\^w$ is rapidly decreasing, only the terms with say $m,m'=n + O(\log N)$ contribute anything non-negligible. 
 Thus 
\begin{equation*}
\int_{G(N)} |S_n(U)|^2 dU \sim \sum_{m,m'=n+O(\log N)} \^w(m-n) \overline{\^w(m'-n)} \int_{G(N)} \tr (U^m) \overline {\tr(U^{m'}) }dU \;.
\end{equation*}

\bigskip

\noindent{\bf  The unitary case $G(N) = U(N)$:}

We use Dyson's lemma \cite{Dyson}
\begin{equation*}
\int_{U(N)} \tr (U^m)\overline{ \tr(U^{m'})} dU = \begin{cases} N^2,&m=m'=0 \\
\delta(m,m') \min(|m|,N),& (m,m')\neq (0,0) .
\end{cases}
\end{equation*}
In particular only the diagonal terms contribute. In our case, $m,m'\sim n$ are nonzero, hence we get 
$$
\int_{U(N)} |S_n(U)|^2 dU \sim \sum_{m=n+O(\log N)} |\^w(m-n)|^2 \min(|m|,N) \;.
$$
Since $m$ varies very little around $n$, we can replace $\min(|m|,N)$ by $\min(n,N)$ with negligible error to obtain
\begin{equation*}
\begin{split}
\int_{U(N)} |S_n(U)|^2 dU &\sim  \min(n,N)\sum_{m=n+O(\log N)} |\^w(m-n)|^2\\
&\sim \min(n,N)\sum_{{\rm all}\; m} |\^w(m)|^2 = \min(n,N)\int_0^1|w(\gamma)|^2d\gamma
\end{split}
\end{equation*}
by Plancherel. 

\bigskip

 \noindent{\bf  The symplectic case $G(N) = {\rm USp}(2g)$:}

The expected values for the symplectic group ($N=2g$) are \cite[Lemma 2]{KO}  

i) If $m=n$ then 
\begin{equation*}\label{RMT diagonal pairs}
  \int_{\USp(2g)}|\tr U^n|^2 dU  = 
  \begin{cases}
    n+\eta(n),& 1\leq n \leq g \\ n-1+\eta(n),& g+1\leq n \leq 2g \\ 2g,& n>2g .
  \end{cases}
\end{equation*}

ii) If $1\leq m<n$
\begin{equation*}\label{RMT generic pairs}
  \int_{\USp(2g)}\tr U^m\tr U^n dU  = 
  \begin{cases}
    \eta(m)\eta(n),& m+n \leq 2g \\ 
\eta(m)\eta(n)-\eta(m+n),& m<n\leq 2g, \quad m+n>2g\\
-\eta(m+n),& n>2g,\quad n-m\leq 2g \\
0,& n-m>2g ,
  \end{cases}
\end{equation*}
and in particular, if $m\neq m'$ (and neither is zero) then  
\begin{equation}\label{rough diagonal}
\int_{\USp(N)} \tr (U^m) \overline {\tr(U^{m'}) }dU = O(1)
\end{equation}
while for $m=m'\neq 0$ we obtain
\begin{equation}\label{rough offdiagonal}
\int_{\USp(N)} |\tr(U^m)|^2dU = \min(m,N) + O(1)
\end{equation}
so that
\begin{multline*}
\int_{\USp(N)} |S_n(U)|^2 dU \sim \sum_{m=n+O(\log N)} |\^w(m-n)|^2  \min(m,N)
\\ + 
\sum_{m,m'=n+O(\log N)} \^w(m-n) \overline{\^w(m'-n)} O(1) \;.
\end{multline*} 
The second term is $O(\log N)$, while the first is as in the unitary case, so that again we recover
$$
\int_{\USp(N)} |S_n(U)|^2 dU \sim \min(n,N)\int_0^1|w(\gamma)|^2d\gamma \;.
$$

For the orthogonal group $G(N) = {\rm SO}(N)$ with $N$ even, we have the same result because \eqref{rough diagonal}, \eqref{rough offdiagonal} are still valid 
(see \cite[Lemma 2]{KO}).  
\end{proof}


\section{A function field model}

\subsection{The group of sectors} 
Our goal in this section is to formulate and prove an analogue of Conjecture~\ref{sharp conj full regime} and of  Conjecture~\ref{conj full regime} in the setting of the ring of polynomials  over a finite field of $q$ elements ($q$ odd), in the limit of large $q$. Using the notation in the Introduction, we denote by\footnote{Katz \cite[\S2]{Katz supereven} denotes 
$B^\times_{\rm even} = \Hev$, and $B^\times_{ \rm odd} =\Sone_{\kk}$. }
$$
\Sone_{\kk}=\{f\in \fq[S]/(S^{\kk}): f(0)=1, \; f(-S)f(S)=1\bmod S^{\kk}\}
$$
the elements of unit norm and constant term $1$ in $\Big( \fq[S]/(S^{\kk})\Big)^\times$, 
and 
$$
\Hev := \Big\{f\in \Big( \fq[S]/(S^{\kk})\Big)^\times : f(-S) = f(S)\bmod S^{\kk}\Big\}
$$
the subgroup of even polynomials. 

\begin{lemma}\label{lemma Katz} \cite[Lemma 2.1]{Katz supereven} 
i) We have a direct product decomposition
$$\Big(\fq[S]/(S^{\kk})\Big)^\times=\Hev\times \Sone_{\kk} .
$$

ii) The order of $\Sone_{\kk}$ is 
\begin{equation*}
 \#\Sone_{\kk} = q^\kappa \;,  
\end{equation*}
where  $\kappa:=\kk -1-\lfloor \frac{\kk-1}2\rfloor = \lfloor \frac k2 \rfloor$,  
so that    
$$
\kk=\begin{cases} 2\kappa+1\\2\kappa .\end{cases} 
$$
\end{lemma}
 \begin{proof}
i) is stated in \cite{Katz supereven}  for $\kk$ even, but the proof is valid for arbitrary $\kk\geq 1$.

ii) The order of $\Hev$ is 
$$\#\Hev=(q-1)q^{\lfloor \frac{\kk-1}2 \rfloor}$$ 
since we can write any element of $\Hev$ as 
$$ h= \sum_{0\leq 2 j< \kk}  h_j S^{2j}=\sum_{j=0}^{\lfloor  \frac{\kk-1}2 \rfloor} h_j S^{2j} \in \Hev,\quad h_0\neq 0$$
and the number of such elements is clearly $(q-1) q^{\lfloor \frac{\kk-1}2 \rfloor}$. 
Since the order of $\Big(\fq[S]/(S^{\kk})\Big)^\times$ is $ (q-1) q^{\kk-1}$, we obtain that 
the order of $\Sone_{\kk}$ is 
\begin{equation*}
 \#\Sone_{\kk} = q^{\kk -1-\lfloor \frac{\kk-1}2\rfloor }=q^\kappa \;,  
\end{equation*}
as claimed.
\end{proof}
 

 We put an absolute value $|f| = q^{-\ord(f)}$  on $ \fq[[S]]$,     where $\ord(f) =\max(j: S^j\mid f)$. 
We then divide $\mathbb S^1$ into ``sectors" 
$$
\Ball(u;\kk) = \{v\in \mathbb S^1: |v-u|\leq q^{-\kk}\}\;.
$$
so that by definition, for $u,v\in \mathbb S^1\subset \fq[[S]]$
\begin{equation}\label{eq sectors}
v\in \Ball(u;\kk) \Leftrightarrow u=v\bmod S^{\kk}
\end{equation}
Consequently, the sectors $\Ball(u;\kk)$ are in bijection with the group $\Sone_{\kk}$, and their number is 
$$K := \#\Sone_{\kk}=q^\kappa\;.
$$

Expanding in $\fq[[S]]$: 
$$u=\sum_{j=0}^\infty u_j S^j, \quad u_0=1$$
and likewise for $v$, we see that 
$v\in \Ball(u;\kk)$ is equivalent to 
$$v_j=u_j, \quad j=1,\dots, \kk-1\;.
$$

%

We have a modular version of the homomorphism $U$ from \eqref{def of map U} 
$$
U_{\kk}: \Big( \fq[S]/(S^{\kk})\Big)^\times \to \Sone_{\kk} , \qquad f\mapsto \sqrt{f/\sigma(f)}  \bmod S^{\kk}
$$
whose kernel is $\Hev$. Note that $f/\sigma(f)\in \Sone_{\kk}$ as it has unit norm and constant term $1$, and in $\Sone_{\kk}$ 
the square root is well defined since $\Sone_{\kk}=q^\kappa$ has odd order.  

\begin{lemma}\label{lem:U_k surjective}
The homomorphism $U_{\kk}: \Big( \fq[S]/(S^{\kk})\Big)^\times \to \Sone_{\kk}$ is surjective. 
\end{lemma}
\begin{proof}
The kernel of $U_{\kk}: \Big( \fq[S]/(S^{\kk})\Big)^\times \to \Sone_{\kk}$ is $\Hev$ because the kernel of $f\mapsto f/\sigma(f)$ is,   by definition, $\Hev$, and the square root map is an automorphism of $\Sone_{\kk}$.  
According to   Lemma~\ref{lemma Katz}(i), the map is therefore onto. 
\end{proof}

\subsection{Super-even characters and their L-functions}
A super-even character modulo $S^{\kk}$ is a Dirichlet character
$$\Xi: \Big( \fq[S]/(S^{\kk}) \Big)^\times \to \C^\times$$
which is trivial on $\Hev$. In particular, $\Xi$ is even (trivial on the scalars $\fq^\times$).  These are the analogues of Hecke characters in \S~\ref{sec:Hecke chars}.  
The group  of super-even characters mod $S^{\kk}$ is the character group of 
$\Big( \fq[S]/(S^{\kk}) \Big)^\times/\Hev\simeq \Sone_{\kk}$. Hence by general orthogonality relations for characters of a finite Abelian group, the super-even characters separate the cosets of $\Hev $, that is the elements of $\Sone_{\kk}$.

\begin{proposition}\label{prop: equiv sector} 
 For $f\in  \Big( \fq[S]/(S^{\kk})\Big)^\times $,  
and $u\in  \Sone_{\kk}$, the following are equivalent: 
\begin{enumerate}
\item  $U_{\kk}(f)\in \Ball(u;\kk)$ 
\item $U_{\kk}(f) = U_{\kk}(u)$
\item  $f \cdot \Hev = u\cdot \Hev $
\item  $\Xi(f)  = \Xi(u)$ for all super-even characters mod $S^{\kk}$. 
\end{enumerate}
\end{proposition} 
\begin{proof}

For $u\in \Sone$ we have $U_{\kk}(u) = \sqrt{u/\sigma(u)}=\sqrt{u^2}=u \bmod S^k$ and so combining with \eqref{eq sectors} we find that $U_{\kk}(f) = U_{\kk}(u)$  is equivalent to $U_{\kk}(f)\in \Ball(u;\kk)$.

According to Lemma~\ref{lem:U_k surjective}, the map $U_{\kk}$ is onto.
Therefore, since the kernel of $ U_{\kk}(u)$ is $\Hev $, we obtain that $U_{\kk}(f) = U_{\kk}(u)$ is equivalent to $f\cdot \Hev = u\cdot \Hev$ in $\Big( \fq[S]/(S^{\kk})\Big)^\times$. 


Using the orthogonality relations for characters of $ \Sone_{\kk}$ (super-even characters) 
we obtain the final equivalence. 
\end{proof}

 The Swan conductor of an even nontrivial character $\Xi$ mod $S^{\kk}$ is the maximal integer $d<\kk$ such that $\Xi$ is nontrivial on the subgroup 
$$\Gamma_d:=\Big( 1+(S^d)\Big)/(S^{\kk}) \subset \Big(\fq[S]/(S^{\kk})\Big)^\times . 
$$ 
Then $\Xi$ is a {\em primitive} character modulo $S^{d(\Xi)+1}$. 
For a super-even character, the Swan conductor is necessarily {\em odd}, since super-even characters are automatically trivial on $\Gamma_d$ for $d$ even. 

Let $\Xi$ be a nontrivial even character modulo $S^{\kk}$. 
The L-function associated to $\Xi$ is:
\begin{equation}\label{euler product for L}
L(z,\Xi) = \sum_{f\;{\rm monic}} \Xi(f)z^{\deg f} = \prod_{P\;{\rm prime}}(1-\Xi(P)z^{\deg P})^{-1}, \quad |z|<1/q , 
\end{equation}
which for nontrivial  even $\Xi$ is a polynomial in $z$ of degree exactly  $d(\Xi)$ (the Swan conductor of $\Xi$), including a trivial zero at $z=1$.  
  Thus we write for any non-trivial super-even character
\begin{equation}\label{spectral inter for L}
L(z,\Xi) = (1-z) \det(I-zq^{1/2}\Theta_\Xi)
\end{equation}
for a unitary  matrix $\Theta_\Xi\in U(N)$  ($N=d(\Xi)-1$). 
 
For any nontrivial super-even character mod $S^{\kk}$, let 
$$
\Psi(\nu;\Xi):=\sum_{\deg f=\nu} \Lambda(f)\Xi(f)
$$
be the sum over all monic polynomials of degree $\nu$, with $\Lambda(f)$ being the von Mangoldt function. 
The Explicit Formula (obtained by comparing the logarithmic derivative of \eqref{euler product for L} and \eqref{spectral inter for L}, see e.g. \cite{KRprime})   shows that for nontrivial 
super-even $\Xi$, the sum over prime powers $\Psi(\nu;\Xi)$ is a sum over zeros of the L-function associated to $\Xi$: 
 \begin{equation}\label{explicit formula} 
\Psi(\nu;\Xi) = -q^{\nu/2}\tr \Theta_\Xi^\nu-1 \;.
\end{equation}

\subsection{A weighted count}
We introduce a weighted count in terms of the von Mangoldt function on $\fq[S]$, defined as $\Lambda(f)=\deg \mathfrak p$ if $ f=c\mathfrak p^j$ for some prime 
$\mathfrak p\in \fq[S]$ and $j\geq 1$ and scalar $c\in \fq^\times$, and $\Lambda(f)=0$ otherwise. Set
$$
\Psi_{ \kk,\nu}(u) = \sum_{ U(f) \in \Ball(u;\kk) } \Lambda(f)\;,
$$ 
the sum over monic $f\in \fq[S]$ with $\deg f =\nu$ and $f(0)\neq 0$. 

  We want to average  over all directions $u\in \Sone_{\kk}$. 
The mean value is 
$$
\E(\Psi_{\kk,\nu})  = \frac 1{q^\kappa} \sum_{u\in \Sone_{\kk}}\Psi_{\kk,\nu}(u)\;.
$$
By definition, the sum is just the sum over all monic $f\in M_\nu$ (with $f(0)\neq 0$), that is 
$$
\E(\Psi_{\kk,\nu})  = \frac 1{q^\kappa} \sum_{\substack{\deg f=\nu\\ f(0)\neq 0}} \Lambda(f) = \frac 1{q^\kappa } (\sum_{ \deg f=\nu }\Lambda(f)-1) =  \frac {q^\nu-1}{q^\kappa }
$$
by the Prime Polynomial Theorem in $\fq[S]$.

  
We use Proposition~\ref{prop: equiv sector} to pick out prime powers lying in a given sector, and obtain a formula for the sum 
$\Psi_{\kk,\nu}(u)$ in terms of super-even characters.

\begin{lemma}\label{expression for psi(u)}
\begin{equation*}
\Psi_{\kk,\nu}(u)-  \frac {q^\nu-1}{q^\kappa }  =  -\frac { q^{\nu/2}}{q^\kappa}   \sum_{ \Xi \neq \Xi_0 }  \overline{\Xi(u)} \tr\Theta_\Xi^\nu-\delta(u,1)+\frac 1{q^\kappa} , 
\end{equation*}
the sum being over all nontrivial super-even characters mod $S^{\kk}$. 
\end{lemma}
\begin{proof}
From Proposition~\ref{prop: equiv sector}  and the orthogonality relations we find 
$$
\frac 1{q^\kappa} \sum_{\Xi \;{\rm super-even}\bmod S^{\kk}}\overline{\Xi(u)} \Xi(f) = 
\begin{cases}
1,& U(f) \in \Ball(u;\kk)\\ 0,&{\rm otherwise}, 
\end{cases}
$$
 which gives
$$
\Psi_{\kk,\nu}(u)=\sum_{\substack{\deg f=\nu\\ U_{\kk}(f)\in \Ball(u;\kk)}}
\Lambda(f) = \frac 1{q^\kappa} \sum_{\Xi \;{\rm super-even} \bmod S^{\kk}}\overline{\Xi(u)} \sum_{\deg f=\nu} \Lambda(f)\Xi(f) ,
$$
with the sum over all monic $f\in \fq[S]$ of degree $\nu$. 
Hence 
\begin{equation}\label{psi in terms of chars}
\Psi_{\kk,\nu}(u)=
 \frac 1{q^\kappa} \sum_{\Xi \;{\rm super-even} \bmod S^{\kk}}\overline{\Xi(u)}  \Psi(\nu;\Xi) . 
\end{equation}

The contribution of the trivial character $\Xi_0$  is 
$$
 \frac 1{q^\kappa} \sum_{\substack{\deg f=\nu\\ f(0)\neq 0} }\Lambda(f) 
= \frac 1{q^\kappa } \Big(\sum_{ \deg f=\nu }\Lambda(f)-1 \Big) =  \frac {q^\nu-1}{q^\kappa }\;.
$$

Inserting the Explicit Formula \eqref{explicit formula}  gives 
\begin{equation*}
\begin{split}
\Psi_{\kk,\nu}(u)-  \frac {q^\nu-1}{q^\kappa } &= -\frac {1}{q^\kappa}   \sum_{\substack{ \Xi  \;{\rm super-even} \bmod S^{\kk}\\ \Xi \neq \Xi_0}}  \overline{\Xi(u)}\Big( q^{\nu/2}\tr\Theta_\Xi^\nu +1\Big)\\
& =  -\frac { q^{\nu/2}}{q^\kappa}   \sum_{\substack{ \Xi  \;{\rm super-even} \bmod S^{\kk}\\ \Xi \neq \Xi_0}}  \overline{\Xi(u)} \tr\Theta_\Xi^\nu-\delta(u,1)+\frac 1{q^\kappa}
\end{split}
\end{equation*}
on using the orthogonality relations in the form 
$$
\frac 1{q^\kappa}\sum_{\Xi\neq \Xi_0} \overline{\Xi(u)} = \delta(u,1)-\frac 1{q^\kappa} \;.
$$
\end{proof}

We use $|\tr \Theta_\Xi^\nu|\leq 2\kappa-2$ for $\Xi\neq \Xi_0$ to obtain
\begin{corollary}\label{cor: asymp for psi} 
As $q\to \infty$, 
$$
\Psi_{\kk,\nu}(u)=\frac {q^\nu}{q^\kappa }  + O\Big(      q^{\nu/2} \Big) .
$$
Hence for $\kappa<\nu/2$, we obtain an asymptotic formula.
\end{corollary}
By a standard argument, this implies  that $\mathcal N_{\kk,\nu}(u) =N/K  + O(q^{\nu/2})$.

\begin{remark}\label{a big hole}
Note that for $\kappa>\nu/2$, it is no longer necessarily the case that $\Psi_{\kk,\nu}(u)\sim\frac {q^\nu}{q^\kappa }$, in fact there may not be any polynomials $g\in \fq[S]$ of degree $\deg g=\nu<2\kappa$ with direction $U(g)\in \Ball(u;k)$. As an example, assume that $k-1$ is odd, and take  
$$
u =\frac{1+S^{k-1}}{1-S^{k-1}}=1+2S^{k-1} \bmod S^k
$$
and suppose that $\deg g=\nu<2\kappa\leq k-1$ satisfies 
$$U(g)\in \Ball(u;k)=\Ball(1+2S^{k-1};k)\;.
$$ 
By Proposition~\ref{prop: equiv sector}, this is equivalent to $g\in (1+2S^{k-1})H_k$. Reducing modulo $S^{k-1}$ gives $g\in H_{k-1}$, so that $g(-S)=g(S)\bmod S^{k-1}$. But $\deg g<k-1$ hence $g(-S)=g(S)$, that  is $g$ is an even polynomial, hence $U(g)=1$. But then $U(g)=1\notin \Ball(1+2S^{k-1};k)$, a contradiction.  
\end{remark}

\subsection{The variance of $\Psi_{\kk,\nu}$}
The variance of $\Psi_{\kk,\nu}$ is
$$
\var(\Psi_{\kk,\nu}) = \frac 1{q^\kappa}  \sum_{u\in \Sone_{\kk}}| \Psi_{\kk,\nu}(u) - \frac {q^\nu-1}{q^\kappa }|^2 \;.
$$
\begin{theorem}\label{thm var psi}
Assume $q$ is odd, and $\kappa\geq 3$, or that $\kappa=2$ and additionally $5\nmid q$. Then as $q\to \infty$, 
$$\Var(\Psi_{\kk,\nu})\sim q^{\nu-\kappa}
 \begin{cases}\nu+\eta(\nu),&1\leq \nu\leq \kappa-1 \\ \nu-1+\eta(\nu),&\kappa\leq \nu\leq 2(\kappa-1)\\ 2\kappa-2,& \nu>2\kappa-2 .
\end{cases} 
$$
\end{theorem}

In other words, if we denote $X=q^{\nu}$ the number of all monics of degree $\nu$, then 
$$
\frac{\Var(\Psi_{\kk,\nu})}{X/K}\sim  \begin{cases} \log_q X  -1 + \eta(\log_qX), & \frac 12 \log_q X + \frac 12 <\log_q K \leq \log_q X 
\\
\\
2 \log_q K-2,& \log_q K\leq  \frac 12 \log_q X + \frac 12   . 
\end{cases}  
$$
This is to be compared with conjecture~\ref{conj full regime}. 
Note that the range $\nu<\kappa$ is the ``trivial regime", where there are more sectors  than directions; in that case the result is elementary, but of little interest.


\begin{lemma}\label{lem: form for variance} 
 $$
\var(\Psi_{\kk,\nu})  = q^{\nu-\kappa}\Big( \frac 1{q^\kappa}  \sum_{  \Xi \neq \Xi_0}  |  \tr\Theta_\Xi^\nu|^2 \Big)
\cdot  \Big(1+O( \kappa q^{-\nu/2})\Big)
$$
the sum over all nontrivial super-even characters mod $S^{\kk}$. 
\end{lemma}
\begin{proof}

Inserting \eqref{psi in terms of chars} we find
\begin{equation*}
\begin{split}
\var(\Psi_{\kk,\nu})  &=  \frac 1{q^\kappa} \sum_{u\in \Sone_{\kk}}  \Big|  \frac 1{q^\kappa} \sum_{\substack{\Xi \;{\rm super-even} \bmod S^{\kk}\\ \Xi\neq \Xi_0}}\overline{\Xi(u)}  \Psi(\nu;\Xi) \Big|^2\\
&=\frac 1{q^{2\kappa}} \sum_{\substack{ \Xi_1,\Xi_2 \;{\rm super-even} \bmod S^{\kk}\\ \Xi_1,\Xi_2\neq \Xi_0}} \Psi(\nu;\Xi_1) \overline{\Psi(\nu;\Xi_2)}
 \frac 1{q^\kappa}  \sum_{u\in \Sone_{\kk}}\overline{\Xi_1(u)}\Xi_2(u) \;.
\end{split}
\end{equation*}

We use the orthogonality relations in the group of super-even characters, which is the character group of $\Sone_{\kk}$: 
$$
\frac 1{q^\kappa} \sum_{u\in \Sone_{\kk}} \overline{\Xi_1(u)}\; \Xi_2(u) = \delta(\Xi_1,\Xi_2)  .
$$
This gives
\begin{equation*} 
\var(\Psi_{\kk,\nu})  =
\frac 1{q^{2\kappa}} \sum_{\substack{ \Xi  \;{\rm super-even} \bmod S^{\kk}\\ \Xi \neq \Xi_0}} |\Psi(\nu;\Xi )|^2 . 
\end{equation*}

Set $c(u) = \delta(u,1)-\frac 1{q^\kappa}$. 
From Lemma~\ref{expression for psi(u)}  we obtain, on denoting by $\ave{\bullet}_{\Sone}$ the average over all $u\in \Sone_{\kk}$, that 
\begin{equation*}
\begin{split}
\var(\Psi_{\kk,\nu})  &= \frac{q^\nu}{q^{2\kappa}} \sum_{\Xi_1\neq \Xi_0}\sum_{\Xi_2\neq \Xi_0} \tr \Theta_{\Xi_1}^\nu \overline{ \tr \Theta_{\Xi_2}^\nu}\ave{\overline{\Xi_1(u)}\; \Xi_2(u)}_{\Sone}\\
&+2\frac{q^{\nu/2}}{q^\kappa} \Re\sum_{\Xi\neq \Xi_0} \tr (\Theta_\Xi^\nu) 
\ave{ \overline{\Xi(u)} c(u)}_{\Sone}  
+  \ave{c(u)^2}_{\Sone}  \;.
\end{split}
\end{equation*}
Using the orthogonality relations, the averages over $u\in \Sone$ are 
\begin{equation*}
\begin{split}
\ave{\overline{\Xi_1(u)}\;\Xi_2(u)}_{\Sone} &= \delta(\Xi_1,\Xi_2)
\\
\ave{\overline{\Xi(u)} c(u)}_{\Sone}   &=\ave{\overline{\Xi(u)} \delta(u,1)}_{\Sone}  - \frac 1{q^\kappa}\ave{\overline{\Xi(u)} }_{\Sone}  \\
&= \frac 1{q^\kappa }\overline{\Xi(1)} - \frac 1{q^\kappa}\delta(\Xi,\Xi_0)
= \frac 1{q^\kappa }
\end{split}
\end{equation*}
since $\Xi\neq \Xi_0$, and
$$
\ave{  c(u)^2}_{\Sone} =\frac 1{q^\kappa}(1-\frac 1{q^\kappa}) \;.
$$
Substituting into our formula gives
\begin{equation*}
\begin{split}
\var(\Psi_{\kk,\nu}) &= q^{\nu-\kappa} \frac 1{q^\kappa}  \sum_{ \Xi \neq \Xi_0} |  \tr\Theta_\Xi^\nu|^2  
\\
&+ 2\frac{q^{\nu/2}}{q^{2\kappa}} \Re \sum_{\Xi\neq \Xi_0} \tr (\Theta_\Xi^\nu) 
+\frac 1{q^\kappa}\Big(1-\frac 1{q^\kappa}\Big) \;.
\end{split}
\end{equation*}
Finally we use $|\tr \Theta_\Xi^\nu|\leq 2\kappa-2$ for $\Xi\neq \Xi_0$ to get our claim.
\end{proof}

Hence we get an inequality (for all $\kappa$ and $\nu$)
\begin{corollary}
\begin{equation*}
\var(\Psi_{\kk,\nu}) \lesssim  q^{\nu-\kappa} (2\kappa-2)^2 \;.
\end{equation*}
\end{corollary}
This is analogous to Theorem~\ref{thm upper bound on var}. To do better, we invoke an equidistribution result for the zeros of these L-functions. 

\subsection{Proof of Theorem~\ref{thm var psi}} 
We use Lemma~\ref{lem: form for variance}.  
We separate the characters according to their Swan conductor, which is necessarily an odd integer $d(\Xi)< \kk$, whose maximal value is $2\kappa-1$ (recall $\kk=2\kappa$ or $2\kappa+1$). Characters with such maximal conductor make up all primitive super-even characters modulo $S^{2\kappa}$.  
As in \cite{KRprime}, the contribution of characters with smaller Swan conductor $d(\Xi)<2\kappa-1$ is negligible, and up to lower order terms one finds 
\begin{equation}\label{express variance via primitive}
\var(\Psi_{\kk,\nu}) \sim  q^{\nu-\kappa}  \frac 1{\#}\sum_ {\substack{ \Xi \;{\rm super-even}\bmod S^{2\kappa}\\ {\rm primitive} }} |  \tr\Theta_\Xi^\nu|^2 
\end{equation}
the average over all primitive super-even characters modulo $S^{2\kappa}$. 

Katz  \cite[Theorem 5.1]{Katz supereven} showed that for any sequence of odd\footnote{In \cite[Theorem 5.1]{Katz supereven} $q$ is allowed to be even for $2\kappa-2\geq 6$.} 
$q\to \infty$, the Frobenii 
$$\{\Theta_\Xi:\Xi \;{\rm primitive\; super-even} \bmod S^{2\kappa}\}
$$ 
become uniformly distributed in the unitary symplectic group ${\rm USp}(2\kappa-2)$ provided $2\kappa-2\geq 4$, and that the same holds for $2\kappa-2=2$ if the $q$ are co-prime to $10$ (i.e. the characteristic of $\fq$  is not $2$ or $5$).  
  Katz's equidistribution theorem  allows us to replace the average over primitive super-even characters in \eqref{express variance via primitive} by the corresponding continuous average over the unitary symplectic group ${\rm USp}(2\kappa-2)$, to get  
$$
\var(\Psi_{\kk,\nu}) \sim  q^{\nu-\kappa} \int_{\rm USp(2\kappa-2)}
|\tr(U^\nu)|^2 dU \;.
$$

The matrix integral equals, for $\nu>0$ \cite[Lemma 2]{KO}, 
$$
 \int_{\rm USp(2\kappa-2)}|\tr(U^\nu)|^2 dU=\begin{cases}
\nu+\eta(\nu),&1\leq \nu\leq \kappa-1 \\ \nu-1+\eta(\nu),&\kappa\leq \nu\leq 2(\kappa-1)\\ 2\kappa-2,& \nu>2\kappa-2
\end{cases}
$$
where $\eta(\nu)=1$ for $\nu$ even, and equals $0$ for $\nu$ odd. This proves Theorem~\ref{thm var psi}.

\subsection{Relation between variance of $\mathcal N_{\kk,\nu}$ and $\Psi_{\kk,\nu}$}
We can now proceed to prove Theorem~\ref{thm: var N pols}, which  follows from Theorem~\ref{thm var psi} once we establish the following relation between  the variance of $\mathcal N_{\kk,\nu}$ and of $\Psi_{\kk,\nu}$: 
\begin{proposition}\label{prop rel N & Psi} 
Under the conditions of Theorem~\ref{thm var psi}, 
$$\var(\mathcal N_{\kk,\nu})\sim \frac 1{\nu^2} \var(\Psi_{\kk,\nu})
$$
as $q\to \infty$. 
\end{proposition}

Let $\mathbf 1_{\Ball(u;\kk)}$ be the indicator function of the sector $\Ball(u;\kk)$. 
We write 
\begin{equation*}
\begin{split}
\Psi_{\kk,\nu}(u)& = \sum_{\deg f=\nu} \Lambda(f)\mathbf 1_{\Ball(u;\kk)}(U(f)) 
\\& =
\nu \sum_{\substack{\deg P=\nu\\ {\rm prime}}} \mathbf 1_{\Ball(u;\kk)}(U(P))  + R_{\kk,\nu}(u)
\\& =\nu \mathcal N_{\kk,\nu}(u) + R_{\kk,\nu}(u)
\end{split}
\end{equation*}
with the sums over monic polynomials, where
$$
R(u) = R_{\kk,\nu}(u) = \sum_{\substack{\deg f= \nu\\ f\;{\rm not\; prime}}} \Lambda(f)\mathbf 1_{\Ball(u;\kk)}(U(f)) \;.
$$
We subtract the expected value of $\Psi$, which is 
$$
\ave{\Psi} = \frac{q^\nu-1}{q^{\kappa}} , 
$$ 
where we write $\ave{\bullet}$ for the average over all sectors $u\in \Sone_{\kk}$. Compare this with the expected value of 
$\mathcal N =\mathcal N_{k,\nu}$, which is 
$$\ave{\mathcal N } = \frac{N}{q^\kappa} = \frac{q^\nu}{\nu q^\kappa} + O\Big(   \frac{q^{\nu/2}}{\nu q^\kappa}\Big)
$$
by the Prime Polynomial Theorem. Therefore  
\begin{equation}\label{rel psi to N}
\Psi_{\kk,\nu}(u)-\ave{\Psi} = \nu \cdot \Big(\mathcal N(u)-\ave{\mathcal N}\Big) + R(u) +O\Big(   \frac{q^{\nu/2}}{ q^\kappa}\Big)  \;.
\end{equation}
We claim that the mean square of $R$  is bounded by 
\begin{lemma} \label{bound for aveR^2} 
\begin{equation*}
\ave{R^2}:=\frac 1{q^\kappa} \sum_{u\in \Sone_{\kk}} R(u)^2 \ll q^{\nu-2\kappa}   + q^{\frac 23\nu-\kappa} \;.
\end{equation*} 
\end{lemma}
This bound is certainly negligible compared to the variance  of $\Psi_{\kk,\nu}$, which by Theorem~\ref{thm var psi}  is of order $q^{\nu-\kappa}$. Using  \eqref{rel psi to N} gives 
 $$
 \Big|  \nu^2 \ave{| (\mathcal N -\ave{N}|^2} - \ave{|\Psi-\ave{\Psi}|^2} \Big| \ll  \ave{R^2}  + O\Big(   \frac{q^{\nu }}{ q^{2\kappa}}\Big)  \;,
$$
and we obtain 
$$\nu^2\var(\mathcal N) = \var(\Psi) + O\Big(q^{\nu-\kappa} (q^{-\kappa} + q^{-\nu/3})\Big) \;.  
$$
Hence by Theorem~\ref{thm var psi}
$$\var(\mathcal N)\sim \frac 1{\nu^2} \var(\Psi)
$$
as $q\to \infty$.

\subsection{Proof of Lemma~\ref{bound for aveR^2} }
To prove Lemma~\ref{bound for aveR^2} we write 
$$
\ave{R^2} = \sum_{\substack{\deg f,\deg g = \nu\\ \;{\rm not\; prime}}} \Lambda(f)\Lambda(g) \ave{  \mathbf 1_{\Ball(u;\kk)}(U(f))  \mathbf 1_{\Ball(u;\kk)}(U(g)) } \;.
$$
We compute
\begin{equation*}
\begin{split}
\ave{\mathbf 1_{\Ball(u;\kk)}(U(f)) \mathbf 1_{\Ball(u;\kk)}(U(g)) } 
&= \frac 1{q^\kappa}\sum_{u\in \Sone_{\kk}} \mathbf 1_{\Ball(u;\kk)}(U(f))  \mathbf 1_{\Ball(u;\kk)}(U(g)) 
\\&= \begin{cases}\frac 1{q^{\kappa}} ,& U(f)=U(g)\bmod S^{\kk}
\\0,&{\rm otherwise.}
\end{cases}
\end{split}
\end{equation*}
By Proposition~\ref{prop: equiv sector}, the condition $U(f)=U(g)\bmod S^{\kk}$ is equivalent to $\Xi(f) = \Xi(g)$ for all super-even characters modulo $S^{\kk}$, that is 
$$
\ave{\mathbf 1_{\Ball(u;\kk)}(U(f))  \mathbf 1_{\Ball(u;\kk)}(U(g)) } =\frac 1{q^\kappa} \cdot  \frac 1{q^\kappa} \sum_{\Xi\;{\rm super-even}\; \bmod S^{\kk}} \overline{\Xi(f)}\;\Xi(g) \;.
$$
Therefore
\begin{equation}\label{R^2 in terms of Xi v2}
\begin{split}
\ave{R^2}& = \sum_{\substack{\deg f,\deg g = \nu\\ \;{\rm not\; prime}}} \Lambda(f)\Lambda(g)  \frac 1{q^{2\kappa}} \sum_{\Xi\;{\rm super-even}\; \bmod S^{\kk}} \overline{\Xi(f)}\;\Xi(g)
\\
&= \frac 1{q^{2\kappa}}  \sum_{\Xi\;{\rm super-even}\; \bmod S^{\kk}}\Big| \sum_{\substack{\deg f= \nu\\{\rm not\; prime}}} \Lambda(f)\Xi(f) \Big|^2 
\\& =
\frac 1{q^{2\kappa}}  \sum_{\Xi\;{\rm super-even}\; \bmod S^{\kk}}\Big|  B(\nu,\Xi)  \Big|^2 ,
\end{split}
\end{equation}
 where
 \[
 B(\nu,\Xi):= \sum_{\substack{\deg f= \nu\\{\rm not\; prime}}} \Lambda(f)\Xi(f) .
 \]
 We will show below that if $\Xi=1$,   then  
 \begin{equation}\label{trivial bd for B}
 B(\nu,1) \ll_\nu q^{\nu/2}   ,
 \end{equation}
 and  if $\Xi\neq 1$, then
 \begin{equation}\label{nontrivial bd for B}
 |B(\nu,\Xi)|\ll_\nu  q^{\nu/3}  .
 \end{equation}
 Assuming \eqref{trivial bd for B} and \eqref{nontrivial bd for B}, we use the expansion \eqref{R^2 in terms of Xi v2} for $\ave{R^2}$, 
 and insert the bounds \eqref{trivial bd for B} for $\Xi=1$, and \eqref{nontrivial bd for B} for $\Xi\neq 1$  to obtain  
 \[
 \ave{R^2}\ll  q^{\nu-2\kappa} + q^{\frac 23 \nu-\kappa}
 \]
 proving Lemma~\ref{bound for aveR^2}.

 It remains to prove \eqref{trivial bd for B} and \eqref{nontrivial bd for B}. We set 
 \[
 A(\nu,\Xi):= \sum_{\substack{\deg P= \nu\\P\;{\rm  prime}}} \nu\;\Xi(P) 
 \]
 so that
 \begin{equation}\label{B in terms of A} 
 B(\nu,\Xi)=\sum_{\substack{\delta\mid \nu\\\delta<\nu}} A(\delta,\Xi^{\nu/\delta})      .
 \end{equation}
 
 The trivial bound for $A(\nu,\Xi)$ is 
 \[
 |A(\nu,\Xi)|\leq A(\nu,1)= \nu \#\{P\;{\rm prime},\deg P=\nu\}\leq q^\nu .
 \]
 This gives \eqref{trivial bd for B}, because
 \[
 B(\nu,1)=\sum_{\substack{\delta\mid \nu\\\delta<\nu}} A(\delta,1)   \leq   \sum_{\substack{\delta\mid \nu\\\delta<\nu}}q^\delta \ll_\nu q^{\nu/2}
 \]
 since the largest divisor $\delta\mid \nu$ which is smaller than $\nu$ is not larger than $\nu/2$. 
 
 If $\Xi\neq 1$ then we  have a better bound:  
 \begin{equation}\label{nontrivial bd for A}
 |A(\nu,\Xi)|\ll_\nu q^{\nu/2},\quad \Xi \neq 1 .
 \end{equation}
Indeed, write $A(\nu,\Xi)=  \Psi(\nu,\Xi)-B(\nu,\Xi)$, 
  and then use  the trivial bound \eqref{trivial bd for B}: $|B(\nu,\Xi)| \ll q^{\nu/2}  $ and  
  \eqref{explicit formula}:   $|\Psi(\nu,\Xi)|\ll q^{\nu/2}$, 
  to obtain \eqref{nontrivial bd for A}.

 Next, we use the expansion \eqref{B in terms of A} of $B(\nu,\Xi)$  to write 
\[
|B(\nu,\Xi)|\leq  \sum_{\substack{ \delta\mid \nu,\; \delta<\nu\\ \Xi^{\nu/\delta}=1}} A(\delta, 1) + 
\sum_{\substack{ \delta\mid \nu, \; \delta<\nu\\ \Xi^{\nu/\delta}\neq1}} |A(\delta, \Xi^{\nu/\delta})| .
\]

To bound the contribution of divisors $\delta$ with $\Xi^{\nu/\delta}=1$, note that the order of $\Xi$ divides $\#\Sone=q^\kappa$, so that if $\Xi\neq 1$ but $\Xi^{\nu/\delta}=1$ then necessarily $p\mid \nu/\delta$, where $q=p^r$ with $p$ an odd prime 
(since $q$ is odd).     
Hence using the trivial bound $A(\delta,1)\leq q^\delta$ gives
\[
\sum_{\substack{ \delta\mid \nu, \; \delta<\nu\\ \Xi^{\nu/\delta}=1}} A(\delta, 1) \leq 
\sum_{\substack{ \delta\mid \nu \\ p\mid \frac{\nu}{\delta} }} q^\delta .
\]
Now if $p\mid \frac \nu\delta$, then $\delta\mid \frac{\nu}{p}$ so $\delta\leq \frac {\nu}{p}$, and we obtain 
\[
\sum_{\substack{ \delta\mid \nu, \; \delta<\nu\\ \Xi^{\nu/\delta}=1}} A(\delta, 1) \ll_\nu q^{\nu/p} .
\]

We bound the contribution  of divisors $\delta$ with $\Xi^{\nu/\delta}\neq 1$,  using \eqref{nontrivial bd for A}, by
\[
\sum_{\substack{ \delta\mid \nu, \; \delta<\nu\\ \Xi^{\nu/\delta}\neq1}} |A(\delta, \Xi^{\nu/\delta})|  \ll_\nu \sum_{\delta\mid \nu, \; \delta<\nu} q^{\delta/2} \ll q^{\nu/4} ,
\]
again using that  the largest divisor $\delta\mid \nu$ which is smaller than $\nu$ is not larger than $\nu/2$. 
Thus we find that for $\Xi\neq 1$, 
\[
|B(\nu,\Xi)| \ll _\nu q^{\nu/p} +q^{\nu/4}
\]
which proves \eqref{nontrivial bd for B} since $p\geq 3$.


\begin{thebibliography}{99}
\bibitem{BSSW}
L. Bary-Soroker, Y. Smilansky, A. Wolf,  {\em On the Function Field Analogue of Landau's Theorem on Sums of Squares}, Finite Fields Appl. 39 (2016) 195--215. 
 
 
 \bibitem{BKS}
  H. M Bui,   J. P. Keating and  D. J.  Smith, {\em On the variance of sums of arithmetic functions over primes in short intervals and pair correlation for L-functions in the Selberg class}. J. Lond. Math. Soc. (2) 94 (2016), no. 1, 161--185.
 
 \bibitem{Dyson}
F.J. Dyson,  {\em Statistical theory of the energy levels of complex systems}, I, II and III. J. Math. Phys. 3, 140--175 (1962).

\bibitem{Harman Lewis}
G. Harman and P. A. Lewis, {\em Gaussian primes in narrow sectors},  Mathematika 48 (2001), no. 1-2, 119--135 (2003). 

\bibitem{Hecke}
E. Hecke, {\em Eine neue Art von Zetafunktionen und ihre Beziehungen zur Verteilung der Primzahlen}. I. , Math. Z. 1 (1918), 357-376. II, Math. Z. 6 (1920), 11--51


\bibitem{IK}
H. Iwaniec and E. Kowalski, 
Analytic number theory. 
American Mathematical Society Colloquium Publications, 53. American Mathematical Society, Providence, RI, 2004. 

\bibitem{Katz supereven}
N. M. Katz, {\em Witt Vectors and a Question of Rudnick and Waxman}. Int. Math. Res. Not. IMRN, Vol. 2016, No. 00, pp. 1–36
doi: 10.1093/imrn/rnw130


\bibitem{KO}
J.P. Keating and B.E. Odgers, 
{\em Symmetry transitions in random matrix theory $\&$ L-functions}. 
Comm. Math. Phys. { 281} (2008), no. 2, 499--528. 

\bibitem{KRprime}
 J. Keating and Z.  Rudnick, {\em The variance of the number of prime polynomials in short intervals and in residue classes}. Int. Math. Res. Not. IMRN 2014, no. 1, 259--288. 

 
\bibitem{Kovalcik}
F. B. Koval'chik, {\em Density theorems for sectors and progressions} 
Lietuvos Matematikos Rinkinys
(Litovskii Matematicheskii Sbornik), Vol. 15, No. 4, pp. 133--151, October-December, 1975. 

 \bibitem{Kubilius 1950}
I. Kubilyus.
The distribution of Gaussian primes in sectors and contours.  
Leningrad. Gos. Univ. U\v c. Zap. Ser. Mat. Nauk 137(19) (1950), 40--52.

\bibitem{Kubilius 1955} 
J. Kubilius. 
{\em On a problem in the n-dimensional analytic theory of numbers. } 
Vilniaus Valst. Univ. Mokslo Darbai. Mat. Fiz. Chem. Mokslu Ser. 4 1955 5--43.

\bibitem{Levy}
P.  L\'evy, {\em Sur la division d'un segment par des points choisis au
hasard}, C.R. Acad. Sci. Paris 208 (1939), 147--149.

\bibitem{Maknis1975}
M. Maknys. 
{\em Zeros of Hecke Z-functions and the distribution of primes of an imaginary quadratic field}, 
  Lith Math J (1975) 15: 140--149. 
  
\bibitem{Maknis1976}
M. Maknys. 
{\em Density theorems for Hecke Z functions and the distribution of primes of an imaginary quadratic field}, 
Lith Math J (1976) 16: 105--110. 

\bibitem{Maknis1977}
M. Maknys. 
{\em Refinement of the remainder term in the law of the distribution of prime numbers of an imaginary quadratic field in sectors.  } 
Lith Math J (1977) 17: 90--93. 
 
\bibitem{PR}
O.  Parzanchevski and  P. Sarnak, {\em Super-Golden-Gates for PU(2)}. 
	arXiv:1704.02106 [math.NT]
	
\bibitem{RS}
Z. Rudnick and P. Sarnak:  {\em  Zeros of principal L-functions and random matrix theory}. A celebration of John F. Nash, Jr. Duke Math. J. 81 (1996), no. 2, 269--322.
 
 
\end{thebibliography}
\end{document}